\documentclass[12pt,a4paper]{article}

\usepackage{graphicx}
\usepackage{amssymb}
\usepackage{amsmath}
\usepackage{epstopdf}

\usepackage{float}

\usepackage{caption}
\usepackage{subcaption}

\usepackage[margin=0.8in]{geometry} % for the margins (for a 12pt document, the default margin is 1.5 in)

\usepackage{subfloat}

\title{A Formula for the Pluricomplex Green Function of the Bidisk}

\author{Jesse J. Hulse}

\date{}

\graphicspath{{Fig/}}

\numberwithin{equation}{section}

\usepackage{pst-plot,booktabs,mathtools}
\usepackage{amsfonts, amsthm, xcolor}
\usepackage{graphicx}

\DeclareMathOperator*{\C}{\mathbb C}
\DeclareMathOperator*{\D}{\mathbb D}
\DeclareMathOperator*{\R}{\mathbb R}

\newtheorem{theorem}{Theorem}[section]

\newtheorem{proposition}[theorem]{Proposition}
\newtheorem{corollary}[theorem]{Corollary}
\newtheorem{lemma}[theorem]{Lemma}

\theoremstyle{definition}
\newtheorem{definition}[theorem]{Definition}

\theoremstyle{remark}

\usepackage{color}

\usepackage{subcaption}

\begin{document}

\maketitle

%\tableofcontents

\begin{center}
Department of Mathematics \\
Syracuse University \\
Syracuse, NY 13244-1150, USA
\end{center}

\vskip 0.5truein
\begin{center}
{\bf Subject Areas}
\end{center}
\begin{center}
  several complex variables, pluripotential theory, analysis
\end{center}

\begin{center}
{\bf Keywords}
\end{center}
\begin{center}
 complex analysis, pluricomplex Green function, Lempert function
\end{center}

\begin{center}
{\bf Abstract}
\end{center}
In this paper, we derive a formula for the pluricomplex Green function of the bidisk with two poles of equal weights. In 2017, Kosi\'nski, Thomas, and Zwonek proved the Lempert function and the pluricomplex Green function are equal on the bidisk, and their description of Lempert function was pivotal in computing the formula for the pluricomplex Green function.  We divide the bidisk into two open regions, where the formula is found explicitly on the first region, and the other region is the union of a family of hypersurfaces. On each hypersurface, the formula is explicit up to a unimodular constant that is the root of a sixth degree polynomial. This derived formula for the bidisk leads to an explicit formula for the Carath\'eodory metric on the symmetrized bidisk up to a fourth degree polynomial. In 2004, Agler and Young found a formula for Carath\'eodory metric for the symmetrized bidisk that involves a supremum over the unimodular constants. The formula derived in this paper matches Agler and Young's formula, but the unimodular constant is determined by a 4th degree polynomial instead of the before mentioned supremum.
\noindent

\vspace{1cm}
\noindent
%\textit{Keywords: }transform pair; harmonic function; mixed boundary value problem.

\vfill\eject

\section{Introduction}
Let $\Omega\subset \C^n$ be a bounded domain and let $P=\{p_1,\dots,p_m\}\subset \Omega$ be a finite set of poles with corresponding weights $v_j:=v(p_j)$ where $v:P\rightarrow (0,\infty)$. A plurisubharmonic function $f$ has a \textbf{logarithmic pole} of weight $v>0$ at a point $a$ if $f(z)-v\log|z-a|\leq O(1)$ as $z\rightarrow a$.
\begin{definition}    
The \textbf{pluricomplex Green function with poles $P$ and weights $v$} is defined as
\begin{align*}
&g_{\Omega}(z,P,v)\\&=\sup\{ f(z): f\in PSH(\Omega,[-\infty,0)), f \text{ has logarithmic poles at each } p_j\in P \text{ of weights }v_j\}.\end{align*}
\end{definition}
\noindent This definition for the pluricomplex Green function was first formulated by Klimek \cite{MK91}, and was extended to finitely many poles by Lelong \cite{PL89}. One may show that $g_{\Omega}$ will be plurisubharmonic, negative on $\Omega$, maximal on $\Omega\setminus \{p_1,\dots,p_m\}$, and will have logarithmic poles at each $p_j\in P$. Lempert showed that when $\Omega$ is a bounded strongly convex domain with 
real analytic (or smooth) boundary, then $g(\cdot,p)$ (one pole) is real analytic \cite{LL81}. A domain is strongly convex if it has a defining function that has a positive-definite Hessian on each real tangent space $T_p(\partial \Omega)$ at each $p\in \partial \Omega$.

Demailly show that if $\Omega$ is hyperconvex, then $g_{\Omega}(\cdot,P,v):\overline{\Omega}\rightarrow [-\infty,0]$ is continuous when one defines $g_{\Omega}(z,P,v)=0$ on $z\in\partial \Omega$, and that \[(dd^c g_{\Omega}(z,P,v))^n=\sum_{j=1}^m v_j^n \delta_{p_j}(z),\]
where $\delta_{p_j}$ is the Dirac measure at $p_j$ and $(dd^c)^n$ is the complex Monge-Amp\`ere operator \cite{JD87}. When $n=1$, the complex Monge-Amp\`ere operator is a scalar multiple of the Laplacian, and when $n\geq 2$, the complex Monge-Amp\`ere operator is nonlinear. Let $\D$ denote the unit disk in $\C$ and let $\D^2$ denote the bidisk in $\C^2$. In this paper, we present a formula for the pluricomplex Green function with 2 poles of equal 
weights in the bidisk. The computations are based on the description of the extremal disks for the Lempert function (defined in the following section) first formulated by Kosi\'nski, Thomas, and Zwonek \cite{KTZ17}. Let
\[m_{z}(\lambda):=\frac{z-\lambda}{1-\overline{z}\lambda},\]
 $(p_1,p_2)$ and $(q_1,q_2)$ be the two poles in $\D^2$ of weight one, and let $(z_1,z_2)$ be the evaluation point for the Green function.  The pluricomplex Green function is invariant under automorphisms $T$ of the bidisk, meaning that $g_{\D^2}((z_1,z_2),(p_1,p_2),(q_1,q_2))=g_{\D^2}(T(z_1,z_2),T(p_1,p_2), T(q_1,q_2))$. This follows from the inequality $g_{\Omega_2}(f(z_1,z_2),f(p_1,p_2), f(q_1,q_2))\leq g_{\Omega_1}((z_1,z_2),(p_1,p_2),(q_1,q_2))$
for holomorphic functions $f\in \text{Hol}(\Omega_1,\Omega_2)$.
Hence, without loss of generality, one may assume that $(z_1,z_2)=(0,0)$.  Kosi\'nski, Thomas, and Zwonek divide $\D^4$ (in Lemma 3 and its proof found in \cite{KTZ17}) into three regions (up to a permutation of coordinates):
\begin{itemize}
    \item \textbf{Region A}: $(p_1,p_2)$ and $(q_1,q_2)$ must satisfy $|p_1|\neq |p_2|$, $|q_1|\neq |q_2|$, $p_1\neq q_1$, $q_2\neq p_2$, and either $|p_2|<|p_1|$, $|q_2|<|q_1|$, and $|m_{q_2/q_1}(p_2/p_1)|\leq |m_{q_1}(p_1)|$ or $p_2=\omega p_1$ and $q_2=\omega q_1$ for some $|\omega|=1$.
    \item \textbf{Region B}:   $(p_1,p_2)$ and $(q_1,q_2)$ must satisfy $|p_1|\neq |p_2|$, $|q_1|\neq |q_2|$, $p_1\neq q_1$, $q_2\neq p_2$ and must not lie in region A.
    \item  \textbf{Region C}: One of the following must be satisfied: $|p_1|= |p_2|$, $|q_1|= |q_2|$, $p_1= q_1$,  or $q_2= p_2$. 
\end{itemize}
Up to a slight abuse of notation, let $g_{\D^2}((z_1,z_2),(p_1,p_2), (q_1,q_2))$ denote the pluricomplex green function with poles $(p_1,p_2)$ and $(q_1,q_2)$ with weights both equaling one. Given that the automorphism of the bidisk $(w_1,w_2)\mapsto (m_{q_1}(w_1), m_{p_2}(w_2))$ maps $(p_1,p_2)$ to $(m_{q_1}(p_1),0)$ and $(q_1,q_2)$ to $(0,m_{p_2}(q_2))$, we may assume the poles are at $(p,0)$ 
and $(0,q)$ with $p\geq 0$ and $q\geq 0$ by setting $p:=m_{q_1}(p_1)$ and $q:=m_{p_2}(q_2)$ and rotating each coordinate appropriately. Next, define the automorphism of the bidisk $T$ as  \begin{align}\label{Tauto}
T(w_1,w_2):=(m_{z_1}(w_1),m_{z_2}(w_2)).
\end{align}
The automorphism is an involution: $T\circ T=id$. Note that
\begin{align*}
    g_{\D^2}((z_1,z_2),(p,0),(0,q))=g_{\D^2}((0,0),(m_{z_1}(p)
,z_2),(z_1,m_{z_2}(q)).\end{align*}
To this end, we will reformulate the three regions above by restricting our poles to $(m_{z_1}(p),z_2)$ and $(z_1,m_{z_2}(q))$ for $(z_1,z_2)\in \D^2$ while still fixing the evaluation point at $(0,0)$. With $(z_1,z_2)\in \D^2$ and $(p,0)$ and $(0,q)$ fixed, we may rewrite the three regions above (up to a permutation of coordinates):
\begin{enumerate}
    \item[] \textbf{Region 1:} $|m_{z_1}(p)|\neq |z_2|$, $|z_1|\neq |m_{z_2}(q)|$, $q\neq 0,p\neq 0$ and $|z_2|<|m_{z_1}(p)|,\\ |m_{z_2}(q)|<|z_1|,$ $|m_{m_{z_2}(q)/z_1}(z_2/m_{z_1}(p))|\leq |m_{z_1}(m_{z_1}(p))|=p \text{ or } z_2=\omega m_{z_1}(p) \\ \text{ and } m_{z_2}(q)=\omega z_1  \text{ for some } |\omega|=1$.
    \item[]\textbf{Region 2:}  $|m_{z_1}(p)|\neq |z_2|$, $|z_1|\neq |m_{z_2}(q)|$, $q\neq 0,p\neq 0$, and must not lie in region 1.
    \item[]\textbf{Region 3:} $|m_{z_1}(p)|=|z_2|$, $|z_1|= |m_{z_2}(q)|$, $p=0$, or $q=0$. This is a closed set with empty interior.
\end{enumerate}
One may show that $(z_1,z_2)$ lies in region 1 if there exists a holomorphic function $f:\D\rightarrow \D$ such that $f(z_1)=z_2$, $f(p)=0$, and $f(0)=q$ or $f(z_2)=z_1$, $f(0)=p$, and $f(q)=0$. See section \ref{region1section} for additional details concerning region 1. If $(z_1,z_2)$ lies in region 2, there is not an interpolating function $f\in \text{Hol}(\D,\D)$ between the coordinates. In section \ref{region1section}, we show that if $(z_1,z_2)$ lies in region 1, then
\[g_{\D^2}((z_1,z_2),(p,0),(0,q))=\log\max\{|z_1m_{z_1}(p)|,|z_2m_{z_2}(q)|\}.\]
The main results of this paper are the formulas presented in Theorem \ref{completedFormulaThm} for region 2 and a corresponding formula for the Carath\'eodory metric of the symmetrized bidisk presented in Theorem \ref{p=qThm}. If $(z_1,z_2)$ lies in region 2, then $g_{\D^2}$ cannot be computed explicitly in all cases, but the formula can be determined up to a unimodular constant $e^{i\theta}$. In section \ref{region2(0,z)section}, a formula is derived for the special case $(z_1,z_2)=(0,z)$. For each $(0,z)$ in region 2, the before mentioned unimodular constant $e^{i\theta}$ plays an important part in determining the formula for the pluricomplex Green function. The unimodular constant $e^{i\theta}$ and $z$ are related by the following polynomial:
\[(e^{i\theta} - z)(\overline{z} e^{i\theta} - 1)(e^{2i\theta} - 1)^2p^2 - (e^{i\theta} - q)(\overline{z} e^{2i\theta} - z)^2(qe^{i\theta} - 1)=0.\]
 For the general case, region 2 is divided into a family of hypersurfaces using the correspondence between $(0,z)$ and $e^{i\theta}$. Theorem \ref{completedFormulaThm} of section \ref{region2section} states that for each $e^{i\theta}\in \mathbb T$, there is a corresponding hypersurface $S(e^{i\theta}$) that is a subset of region 2, and on which the formula for Green function is \footnotesize
\begin{equation*}
\begin{split}
\\&g_{\D^2}((z_1,z_2),(p,0),(0,q))=\\&\log\Big|\frac{-e^{i\theta} p^2z_1z_2 -e^{i\theta}pq + e^{i\theta}pz_2 +e^{i\theta}qz_1 + p^2qz_1z_2 + \omega( p^2q -  p^2z_2 -  pqz_1 + p z_1 z_2) - qz_1z_2}{\omega e^{2i\theta}( p^2 z_1  -  p) + \omega^2( e^{i\theta} p^2 - e^{i\theta} p z_1- p^2q +  pqz_1) + e^{i\theta}(e^{i\theta} q  - e^{i\theta}p q z_1+ p qz_1z_2  - qz_2)+\omega pqz_2(1- pz_1)  }\Big|,
\end{split}
\end{equation*}
\normalsize
where
\begin{align*}
\omega =\frac{-pe^{2i\theta}+p\pm\sqrt{p^2e^{4i\theta} - 4qe^{3i\theta}+ (4- 2p^2 + 4q^2)e^{2i\theta}    -4qe^{i\theta}+p^2}}{2(q-e^{i\theta})}
\end{align*}
and the sign of the square root depends on $e^{i\theta}$. A precise description of the hypersurfaces $S(e^{i\theta})$ can be found in section \ref{hypersurfacessection}. We also find an explicit formula for $g_{\D^2}$ on the three hypersurfaces associated with the unimodular constants $\pm 1, i$. Corollary \ref{formulaS1} of section \ref{sectionhypersurface} shows that the formula on $S(1)$ is 
\begin{align*}
    g_{\D^2}((z_1,z_2),(p,0),(0,q))=\log\Big| \frac{pqz_1z_2 - pz_1z_2 - qz_1z_2 - pq + pz_2 + qz_1}{pqz_1z_2 - pq - pz_1 - qz_2 + p + q}\Big|.
\end{align*}
Points $(z_1,z_2)$ lie on the hypersurface $S(1)$ if there exists a real number $r$ such that $-p<r<\min\{1,p/q\}$ 
and
\begin{align*}
r=\frac{p(pz_1z_2^2 + qz_1^2z_2 - 2pz_1z_2 - 2qz_1z_2 + pz_1 + qz_2 - z_1^2 + 2z_1z_2 - z_2^2)}{pqz_1^2z_2^2 - 2pqz_1z_2 - pz_1^2z_2 - qz_1z_2^2 + 2pz_1z_2 + 2qz_1z_2 + pq - pz_2 - qz_1}.
\end{align*}
Lastly, in Theorem \ref{p=qThm}, we find a formula for the Carath\'eodory metric for the symmetrized bidisk. The proper holomorphic map $\pi(z_1,z_2)=(z_1+z_2,z_1z_2)$ provides a way to derive the formula for the Carath\'eodory metric of the symmetrized bidisk from the formula for the pluricomplex Green function in the special case where the poles are $(p,0)$ and $(0,p)$. Agler and Young found a formula for Carath\'eodory metric of the symmetrized bidisk that involves a supremum over the unimodular constants \cite{AY04}. The formula derived in this paper matches Agler and Young' formula, but the unimodular constant is determined by a 4th degree polynomial instead of the before mentioned supremum. Further, Trybula finds a formula for the infinitesimal Carath\'eodory metric for the symmetrized bidisk which requires a root of a fourth degree polynomial \cite{MT14}, mirroring the need for the sixth degree polynomial found in Proposition \ref{fundpolyprop} and a fourth degree polynomial found in Theorem \ref{p=qThm}.

\section{Preliminaries} \label{prelimsection}
As before, let $P=\{p_1,\dots,p_m\}\subset \Omega
$ be a finite set of poles and let $v_j:=v(p_j)$ be a positive weight for $p_j$ where $v:P\rightarrow (0,\infty)$. Next, let $Q=\{p'_1,\dots,p'_k\}\subseteq P$ and let $v_Q=\{ v_1',\dots,v_k'\}$ be the set of the corresponding weights. 
\begin{definition}
Define
\[L_{\Omega}(z,Q,v_Q)\coloneqq\inf\Big\{\sum_{j=1}^k v_j'\log|\lambda_j'|: f\in \text{Hol}(\D, \Omega), f(0)=z, f(\lambda_j')=p'_j,\lambda'_j\in \D\Big\}.\]
Next, define the \textbf{Lempert function for $\Omega$ with poles $P$ and weights $v$} as:
\[l_\Omega(z,P,v)\coloneqq\min\{ L_{\Omega}(z,Q,v_Q):Q\subseteq P\}.\]
\end{definition}
\noindent We say a disk $f:\D\rightarrow \Omega$ is \textbf{extremal} for $z$ and poles $P$ if $f(0)=z$, $f(\lambda_j')=p'_j$, $p'_j\in Q$ and $l_{\Omega}(z,P,v)=\sum_{j=1}^k v_j'\log|\lambda_j'|$.
\begin{definition}   
Define the \textbf{Carath\'eodory function}  as follows:
\begin{align*}
        c_{\Omega}(z,P)\coloneqq \sup\{ \log|F(z)|: F\in \text{Hol}(\Omega, \D), F(p_j)=0, p_j\in P\}.
    \end{align*}
    \end{definition}
\noindent The Carath\'eodory function is related to the \textbf{Carath\'eodory distance for }$\Omega$, which is defined as
    \[C_{\Omega}(z_1,z_2)\coloneqq\sup\{ d(f(z_1),f(z_2)):f\in \text{Hol}(\Omega,\D)\},\]
    where $d(z_1,z_2)=|(z_1-z_2)/(1-\overline{z_1}z_2)|$ is the hyperbolic distance on $\D$. We note that the Carath\'eodory distance is usually defined with $\tanh^{-1}$, but for the purposes of this work we will omit the $\tanh^{-1}$.
    Next, define
    \[\delta_{\Omega}(z_1,z_2)=\inf\{ d(\lambda_1,\lambda_2): \exists h\in \text{Hol}(\D,\Omega) \text{ such that } f(\lambda_1)=z_1,f(\lambda_2)=z_2\}.\]
    Then the \textbf{Kobayashi distance} on $
    \Omega$, denoted by $K_{\Omega}(z_1,z_2)$, is defined to be the largest pseudodistance on $\Omega$ dominated by $\delta_\Omega$. One may show that $C_{\Omega}\leq K_{\Omega}\leq \delta_{G}$.
    In a similar manner, for general bounded domains $\Omega$ and poles having weights equal to one, one may prove that
    \begin{align}\label{cgl ineq}
    c_{\Omega}(z,P)\leq g_{\Omega}(z,P)\leq l_{\Omega}(z,P).\end{align}
    Lempert showed that $g_{\Omega}(z,p)=l_{\Omega}(z,p)$ for one pole $p$ in a bounded convex domain $\Omega$ in $\C^n$ \cite{LL81,LL83}. Coman  proved that the pluricomplex Green function and the Lempert function with 
two poles of equal weights are equal on the unit ball in $\C^n$ \cite{DC00}.  He also found the formula for the pluricomplex Green function with two poles of equal weight for the unit ball in $\C^2$ and showed that Green function was of class $C^{1,1}$ but not of class $C^{2}$.
It seems that equality of the Lempert function and Green function is a rather special situation. The first example of the two being unequal was found in 2003 by Carlehed and Weigerinck for the bidisk with two poles of unequal weights \cite{CW03}. Another example was found by Thomas and Trao for the bidisk with four poles of equal weights \cite{TT03}. 

Kosi\'nski, Thomas, and Zwonek proved the equality between the Lempert function and the Carath\'eodory function with two poles of weight one in the bidisk by solving certain Nevanlinna-Pick interpolation problems from bidisk to the unit disk as explained below \cite{KTZ17}. Kosi\'nski later showed that the equality between the Lempert and Carath\'eodory functions with two poles of weights one in $\D^n$ and showed the equality does not hold for three poles \cite{KO15}.  Kosi\'nski, Thomas, and Zwonek map the evaluation point to $(0,0)$ and let the poles $(p_1,p_2)$ and $(q_1,q_2)$ be two arbitrary points in the bidisk. They show that the extremal disks for the Lempert function with poles in region 2 are of the form
\begin{align}\label{deg2disk}
    \phi(\lambda)=(\lambda m_{\alpha}(\lambda),\omega \lambda m_{\beta}(\lambda)),
    \end{align}
where $\alpha,\beta,c\in\D$,\hspace{1mm}$\phi(c)=(p_1,p_2),\hspace{1mm} \phi(m_{\gamma}(c))=(q_1,q_2),$
and $|\omega|=1$, $\gamma=t\alpha+(1-t)\beta$, $t\in (0,1)$. They discovered the form of the functions $F\in \text{Hol}(\D^2,\D)$ such that $F\circ \phi=m_cm_d$ or equivalently the form of the functions $G:=m_{cd}\circ F$ such that
\begin{align*}
    G(\lambda m_\alpha(\lambda),\omega \lambda m_\beta(\lambda))=\lambda m_\gamma(\lambda),\hspace{5mm} \lambda \in \D.
\end{align*}

\noindent They do so by explicitly solving the following Pick interpolation problem from the bidisk to the disk:

\begin{align*}
    \begin{cases}
        (0,0)\mapsto 0,\\
        (\gamma m_\alpha(\gamma),\omega \gamma m_\beta(\gamma))\mapsto 0,\\
        (\lambda'm_\alpha(\lambda'), \omega \lambda'm_\beta(\lambda'))\mapsto \lambda'm_\gamma(\lambda').
    \end{cases}
\end{align*}  

\noindent This type of interpolation problem was studied by
Agler and McCarthy (pages 197-206) \cite{AM02}. This problem has a unique solution given by
\[G(z_1,z_2)=\frac{t z_1+(1-t)\overline{\omega}z_2+\tau \overline{\omega}z_1z_2}{1+\tau((1-t)z_1+t\overline{\omega}z_2)}\]
with $\tau=(\overline{\alpha-\beta})/(\alpha-\beta)$ and $t$ as before. Equality with the pluricomplex Green function follows from (\ref{cgl ineq}). The formula for the pluricomplex Green function is then $g_{\D^2}((0,0),(p_1,p_2),(q_1,q_2))=\log|cm_{\gamma}(c)|$ for $(p_1,p_2)$ and $(q_1,q_2)$ in the second region given that $\phi(c)=(p_1,p_2)$ and $\phi(m_\gamma(c))=(q_1,q_2)$. We find a value for the pluricomplex Green at $(z_1,z_2)$ (and fixed poles $(p,0)$ and $(0,q)$) by first finding the corresponding $c$ and $m_c(\gamma)$ for the special case that $(z_1,z_2)=(0,z)$ in terms of the data and a unimodular $e^{i\theta}$. We find a formula on all of region 2 by constructing hypersurfaces $S(e^{i\theta})$ and determining a formula on each hypersurface.

\section{Region One}\label{region1section}
The invariant form of the Schwarz lemma states that for $f\in \text{Hol}(\D,\D)$, for all $\lambda_1,\lambda_2\in\D$,
\begin{align*}
    |m(f(z_1),f(z_2))|=\Big|\frac{f(\lambda_1)-f(\lambda_2)}{1-\overline{f(\lambda_1)}f(\lambda_2)} \Big|\leq \Big| \frac{\lambda_1-\lambda_2}{1-\overline{\lambda_2}\lambda_1}\Big |  =|m(z_1,z_2)|.
\end{align*}
Further, if 
\begin{align*}
     |m(z_1,z_2)|=\Big|\frac{z_1-z_2}{1-\overline{z_1}z_2} \Big|\leq \Big| \frac{\lambda_1-\lambda_2}{1-\overline{\lambda_2}\lambda_1}\Big | =|m(\lambda_1,\lambda_2)|,
\end{align*}
then there exists a holomorphic function $f:\D\rightarrow \D$ such that $f(\lambda_1)=z_1$ and $f(\lambda_2)=z_2$. Next, consider the following 3-point holomorphic interpolation problem. Find a function $f\in \text{Hol}(\D,\D)$ such that $f(\lambda_j)=z_j$ for $1\leq j\leq 3$ where $\lambda_1,\lambda_2,\lambda_2\in \D$ and $z_1,z_2,z_3\in \D$.  The Nevanlinna-Pick Interpolation Theorem states that there will exist such a function $f$ if the matrix
\begin{align*}
    \Big[\frac{1-z_j\overline{z_k}}{1-\lambda_j\overline{\lambda_k}}\Big]_{j,k=1}^3
\end{align*}
is positive semi-definite. See \cite{GMR18} for more about the Nevanlinna-Pick Interpolation Theorem. One may also view this interpolation problem as follows. By composing with automorphisms of the unit disk, we may assume without loss of generality that $\lambda_3=z_3=0$. Let $|z_1|<| \lambda_1|$ and $|z_2|<|\lambda_2|$. If $|m(z_1/\lambda_1,z_2/\lambda_2)|\leq |m(z_1,z_2)|$, then there will exist a $g\in \text{Hol}(\D,\D)$ such that $g(\lambda_1)=z_1/\lambda_1$ and $g(\lambda_2)=z_2/\lambda_2$. If so, then the function $\lambda \mapsto \lambda g(\lambda)$ will be the desired interpolating function $f$. Thus we see that region 1 consists of the $(z_1,z_2)\in \D^2$ such that there is a holomorphic function $f\in \D$ such that $f(z_1)=z_2$, $f(p)=0$, and $f(0)=q$ or $f(z_2)=z_1$, $f(0)=p$, and $f(q)=0$.
      If $f(z_1)=z_2$, $f(p)=0$, and $f(0)=q$ (maps the first coordinate to the second), then the extremal disks for the Lempert function found in \cite{KTZ17} are
\[\phi(\lambda)=(m_{z_1}(\lambda),f(m_{z_1}(\lambda)).\]
The extremal functions for the Carath\'eodory function found in \cite{KTZ17} will be of the form
\[F(w_1,w_2)=w_1m_{p}(w_1).\]
Is is immediate that $\phi(0)=(z_1,z_2), \phi(m_{z_1}(p))=(p,0), \phi(z_1)=(0,q)$. Also $F(p,0)=F(0,q)=0$ and $F(z_1,z_2)=z_1m_p(z_1)$. Thus $g_{\D^2}((z_1,z_2),(p,0),(0,q))=\log|z_1m_p(z_1)|$. The steps above may be repeated with $f(z_2)=z_1$, $f(0)=p$, and $f(q)=0$ ($f$ maps the second coordinate to the first). In this case $g_{\D^2}((z_1,z_2),(p,0),(0,q))=\log|z_2m_q(z_2)|$.
If $(z_1,z_2)$ lies in region 1, and if $f$ maps the first coordinate to the second, then $|z_1|\geq |z_2|$ and  $|m_{z_1}(p)|\geq |m_{z_2}(q)|$ by the conditions for $(z_1,z_2)$ to be in region one and the fact that $f$ maps the first coordinate to the second. Thus in this case $\max\{|z_1m_{z_1}(p)|,|z_2m_{z_2}(q)|\}=|z_1m_{z_1}(p)|$. If $f$ maps the second coordinate to the first, then again by the conditions needed for $(z_1,z_2)$ to be in region 1, we have  $|z_2|\geq |z_1|$ and  $|m_{z_2}(q)|\geq |m_{z_1}(p)|$.  Thus in this case $\max\{|z_1m_{z_1}(p)|,|z_2m_{z_2}(q)|\}=|z_2m_{z_2}(q)|$. If $z_2 =\omega m_{z_1}(p)$ and $m_{z_2}(q)=\omega z_1$ where $|\omega|=1$, then $|z_1m_{z_1}(p)|=|z_2m_{z_2}(q)|$. Thus the formula for $(z_1,z_2)$ in region 1 is 
\[g_{\D^2}((z_1,z_2),(p,0),(0,q))=\log\max\{|z_1m_{z_1}(p)|,|z_2m_{z_2}(q)|\}\]

\section{Region 2: the Case When $z_1=0$}\label{region2(0,z)section}
Given the point $(z_1,z_2)$ and poles $(p,0)$ and $(0,q)$, we map $(z_1,z_2)$ to the origin via the automorphism given in (\ref{Tauto}) and consider the new triple: $(0,0),(m_{z_1}(p),z_2),(z_1,m_{z_2}(q))$. Trying to find $c,t,$ and $\omega$ in terms of the data alone is extremely difficult. Rather, we will construct, up to a unimodular constant $e^{i\theta}$, the extremal disks for the special case when $z_1=0$, and then use these disks to construct the hypersurfaces $S(e^{i\theta})$ in section \ref{region2section}.\par 
In this section, we construct the extremal disks, up to a unimodular constant $e^{i\theta}$, for the set of points $(0,0),(p,z_2),(0,m_{z_2}(q))$. For simplicity, set $z_2=z$. There is a one parameter family of degree 2 Blaschke products that fix 0 and map a given $c\in \D$ to a point $z\in \D$ (assuming $|z|<|c|$). Explicitly, this family is described by
\[\{\lambda\mapsto\lambda m_{z/c} (\tau m_c(\lambda)): |\tau|=1\}.\]
\subsection{General Computations}\label{generalComp}
Let $q>0$. For now, we assume that $p$ is complex, but later we will show that $p$ can taken to be real without any computational complications.  For the rest of this section, let $\alpha,\beta,c,\omega$, and $t$ be the parameters associated with the extremal disk for the triple $(0,0),(p,z)$ and $(0,m_z(q))$. \par  
In order to lie in region 2, we need  $|p|\neq |z|$, $0\neq |m_z(q)|$ (i.e. $z\neq q$), $p\neq 0$, $m_z(p)\neq z$, and one of following 3 conditions not to hold:
\begin{itemize}
    \item $0<|m_z(q)|$
    \item $|p|<|z|$
    \item $|m_{0}(\frac{p}{z})|<|m_{z}(m_{z}(q))|=|q|$.
\end{itemize}
These three inequalities can be rewritten as $q\neq z$, $|z|>|p|$, and $q>|p/z|$. Hence the conditions to be in region two are
    $|p|\neq |z|, z\neq q, p\neq 0, m_z(p)\neq z$,
and one either $|z|<|p|$ or $|z|<|p|/q$, or equivalently
\begin{align*}
    |z|<\min\{ |p|/q,|p|,1\}=\min\{|p|/q,1\}.
   \end{align*}
Recall that we are looking for $\alpha, \beta,c\in \D$, $t\in(0,1)$, and $\omega \in \mathbb T$ such that $\phi(c)=(p,z)$ and $\phi(m_{\gamma}(c))=(0,m_z(q))$ with $\gamma=t \alpha+(1-t)\beta$. If $m_\gamma(c)=0$, then $\phi(m_\gamma(c))=(0,0)$, which cannot happen. Thus $\alpha=m_\gamma(c)$.  The conditions above may be restated as
\begin{align}
    &\alpha =m_\gamma(c), \label{blaschkecon1}\\& cm_\alpha(c)=p,\label{blaschkecon2}\\& \omega c m_{\beta}(c)=z,\label{blaschkecon3}\\& \omega \alpha m_{\beta}(\alpha)=m_z(q),\label{blaschkecon4}\\
    & t\in (0,1).\label{blaschkecon5}
\end{align}

\begin{proposition}\label{systemsprop}
    Write $c=rv$ where $r>0$ and $|v|=1$.
    Conditions (\ref{blaschkecon1})-(\ref{blaschkecon5}) imply the following: 
    \begin{align}
        & t= \frac{(|p|^2 - 1)(|c|^2\overline{c}\omega p - \overline{c}|p|^2z - c|c|^2\omega + |p|^2c\omega - \overline{c}\omega p + cz)(\overline{c}-\overline{z}c\omega  )}{P},\label{systemcon1}\\& 
         \alpha=\frac{\overline{c} |c^2|p - |c|^2c + c|p|^2 - \overline{c}p}{|c|^2(|p|^2 - 1)}, \label{systemcon2}\\
        &\beta=\frac{\overline{c}|c|^2z - c|c|^2\omega + |z|^2c\omega- \overline{c}z}{\omega |c|^2(|z|^2 - 1)},\label{systemcon3}\\
        & 
        -\overline{z}\omega^2v^4 + \overline{p}\omega v^4 + \overline{z}\omega^2 pv^2 - \overline{p}zv^2 - \omega p + z=0,\label{systemcon4}
        \\&  |c|^2=r^2=\frac{p(1-\overline{p}v^2)(|p|^2\omega v^2 - \overline{z}\omega qv^2 + \overline{z}\omega pq - |p|^2z - \omega p + z)}{(|p|^2\omega v^2 - \overline{z}\omega q v^2 + \overline{z}\omega pq - |p|^2q - \omega p + q)(p-v^2 )}>0. \label{systemcon5}
    \end{align}
    where 
\begin{align*}
    &P=-|c|^4|p|^2\overline{z}\omega^2p + |c|^4|pz|^2\omega + |c|^2|p|^2\overline{z}c^2\omega^2 - |p|^4\overline{z}c^2\omega^2 + \overline{c}^2|c|^2|p|^2\omega p -|c|^2\overline{c}^2p|z|^2\omega  \\&+ \overline{c}^2|p|^2|z|^2\omega p + |c|^4\overline{z}\omega^2p - |c|^2\overline{p}|z|^2c^2\omega + |c|^2|p|^2\overline{z}\omega^2p + \overline{p}|p|^2|z|^2c^2\omega - \overline{c}^2|p|^4z - 2|c|^4|p|^2\omega \\&+ |c|^4|z|^2\omega + 2|c|^2|p|^4\omega - 3|c|^2|pz|^2\omega - |c|^2\overline{z}c^2\omega^2 + |p|^2\overline{z}c^2\omega^2 - 2\overline{c}^2|p|^2\omega p + |c|^2\overline{p}c^2\omega\\& - |c|^2\overline{z}\omega^2p - |p|^2\overline{p}c^2\omega + 2\overline{c}^2|p|^2z + c|^2|z|^2\omega + \overline{c}^2\omega p - \overline{c}^2z.
\end{align*}
\end{proposition}
\begin{proof}
The equivalences (\ref{blaschkecon2}) $\Longleftrightarrow$ (\ref{systemcon2}) and (\ref{blaschkecon3}) $\Longleftrightarrow$ (\ref{systemcon3}) are computable by hand.  The algebraic manipulation capabilities of Maple was used to compute the following expressions. Solving for $t$ from (\ref{blaschkecon1}) gives
\begin{align}
    t = \frac{-\alpha\overline{\beta}c + \alpha - \beta + c}{|\alpha|^2c - \alpha\overline{\beta}c + \alpha - \beta}.\label{texprab}
\end{align}
Substituting in the expressions found for $\alpha$ and $\beta$ found in (\ref{systemcon2})
 and (\ref{systemcon3}) into (\ref{texprab}) and simplifying gives
 \begin{align*}
 t =\frac{(|p|^2 - 1)(|c|^2\overline{c}\omega p - \overline{c}|p|^2z - c|c|^2\omega + |p|^2c\omega - \overline{c}\omega p + cz)(\overline{c}-\overline{z}c\omega  )}{P},\end{align*}
 which is (\ref{systemcon1}).  Next, taking the conjugate of (\ref{texprab}) gives
 \begin{align}
 t-\overline{t}=-\frac{(|\alpha|^2 - 1)(\alpha \overline{\beta}|c|^2 - \beta\overline{\alpha}|c|^2 - \alpha\overline{ c} + \beta\overline{ c} + \overline{\alpha}c - \overline{\beta}c)}{\big||\alpha|^2c - \overline{\beta}\alpha c + \alpha - \beta\big|^2}.\label{tminusbt}
 \end{align}
 Since $t\in(0,1)$, then it must be that (\ref{tminusbt}) is equal to zero. Further, $|\alpha|\neq 1$, so we may conclude that
 \begin{align}\label{tzeroreq}
 \alpha \overline{\beta}|c|^2 - \beta\overline{\alpha}|c|^2 - \alpha\overline{ c} + \beta\overline{ c} + \overline{\alpha}c - \overline{\beta}c=0.
 \end{align}
 Writing $c=vr$ and substituting the expressions found for $\alpha$ and $\beta$ in (\ref{systemcon2}) and (\ref{systemcon3}) into equation (\ref{tzeroreq}) gives
 \begin{align*}
   &t-\overline{t}=0\Longleftrightarrow (r-1)^2(r+1)^2(-\overline{z}\omega^2v^4 + \overline{p}\omega v^4 + \overline{z}\omega^2 pv^2 - \overline{p}zv^2 - \omega p + z)=0.
\end{align*}
Since $|c|=r\neq -1$ or 1, we must have that
\[-\overline{z}\omega^2v^4 + \overline{p}\omega v^4 + \overline{z}\omega^2 pv^2 - \overline{p}zv^2 - \omega p + z=0.\]
This is (\ref{systemcon4}). To show (\ref{systemcon5}), again write $c=rv$  ($r>0, |v|=1$). Substituting in the expressions for $\alpha$ and $\beta$ given in (\ref{systemcon2}) and (\ref{systemcon3}) into (\ref{blaschkecon4}) and solving for $r^2$ gives
    \begin{align*}
        r^2=\frac{p(1-\overline{p}v^2)(|p|^2\omega v^2 - \overline{z}\omega qv^2 + \overline{z}\omega pq - |p|^2z - \omega p + z)}{(|p|^2\omega v^2 - \overline{z}\omega q v^2 + \overline{z}\omega pq - |p|^2q - \omega p + q)(p-v^2 )}>0.
    \end{align*}
    \end{proof}
 \begin{proposition}\label{fundpolyprop}
     \noindent Write $c=vr$ where $|v|=1$ and $r>0$. Conditions (\ref{blaschkecon1})-(\ref{blaschkecon5}) imply that there exists $e^{i\theta}=\omega v^2\in \mathbb{T}$ such that 
    \begin{align}\label{fundamentalpoly}
    (e^{i\theta} - z)(\overline{z} e^{i\theta} - 1)(e^{2i\theta} - 1)^2|p|^2 - (e^{i\theta} - q)(\overline{z} e^{2i\theta} - z)^2(qe^{i\theta} - 1) =0.
       \end{align}
        In particular, if $e^{i\theta}=\pm 1$, then $z$ is real, and if $e^{i\theta}\neq \pm 1$, then
        \begin{align}
        p= \frac{(e^{i\theta} q  - 1)(\overline{z}e^{2i\theta} - z)}{\omega (e^{2i\theta} - 1)(\overline{z}e^{i\theta} - 1)}\text{ or equivalently }\text{Arg}(\omega)=\text{Arg}\Bigg( \frac{(e^{i\theta} q  - 1)(\overline{z}e^{2i\theta} - z)}{p(e^{2i\theta} - 1)(\overline{z}e^{i\theta} - 1)}\Bigg).\label{pexpr}
    \end{align}
       \end{proposition}
    \begin{proof}
    Proposition \ref{systemsprop} shows (\ref{blaschkecon1})-(\ref{blaschkecon5}) imply (\ref{systemcon1})-(\ref{systemcon5}).
    Solving for $\overline{p}$ in condition (\ref{systemcon4}) gives:
    \begin{align}
    \overline{p} =-\frac{(-\overline{z}\omega^2 v^4 + \overline{z}\omega^2pv^2 - \omega p + z)}{v^2(\omega v^2 - z)}.\label{bpexpr1}
    \end{align}
    Substituting this expression above for $\overline{p}$ in (\ref{systemcon5})
    and simplifying gives
        \begin{align}    
    r^2=\frac{p\omega(\overline{z}\omega^2pv^2 - \overline{z}\omega qv^2 - \omega p + z)}{(\omega^2 p v^2 - \omega q v^2 - \omega p q + qz)}.\end{align}
    Take the conjugate  of the expression for $r^2$ given in (\ref{systemcon5}) and substituting in the expression for $\overline{p}$ in (\ref{bpexpr1}). Taking the expressions for $r^2$ and $\overline{r}^2$ and simplifying the expression $r^2-\overline{r}^2=0$ gives:
    \begin{equation}
    \begin{split}
        &0=q(\omega p - z)(\overline{z}\omega^2pv^2 - \omega v^2 - \omega p + z)\\& \cdot(\overline{z}\omega^4pv^6 - \overline{z}\omega^3qv^6 - \omega^3 pv^4 + \overline{z}\omega^2v^4 - \overline{z}\omega^2pv^2 + \omega qv^2z + \omega p - z).
        \end{split}
    \end{equation}
    Because we are in region 2, then $\omega p-z\neq 0$. Solving $\overline{z}\omega^2pv^2 - \omega v^2 - \omega p + z=0$ for $p$ gives
    \[p=\frac{\omega v^2 - z}{\omega(\overline{z}\omega v^2 - 1)}.\]
    But this would imply that $p$ is unimodular. So $\overline{z}\omega^2pv^2 - \omega v^2 - \omega p + z\neq 0$. Thus
    \begin{align}\label{factor1} 
   \overline{z}\omega^4pv^6 - \overline{z}\omega^3qv^6 - \omega^3 pv^4 + \overline{z}\omega^2v^4 - \overline{z}\omega^2pv^2 + \omega qv^2z + \omega p - z=0.\end{align}
    Set $e^{i\theta}=\omega v^2$.
    Temporarily assume that $e^{i\theta}\neq \pm1$.  
    Solving for $p$ in equation (\ref{factor1}) gives
    \begin{align}\label{pexr u}
    p=\frac{(\overline{z}\omega^2v^4 - z)(\omega qv^2 - 1)}{\omega(\omega v^2 - 1)(\omega v^2 + 1)(\overline{z}\omega v^2 - 1)}.
    \end{align}
    Next, take the expression for $p$ given in (\ref{pexr u}) and substitute into the expression for $\overline{p}$ given in (\ref{bpexpr1}). Solving for $\overline p$ and simplifying gives
    \begin{align}\label{bpexpr u}\overline{p}=\frac{\omega(\omega v^2 - q)(\overline{z}\omega^2v^4 - z)}{(\omega v^2 - z)(\omega v^2 - 1)(\omega v^2 + 1))}.
    \end{align}
    One obtains (\ref{fundamentalpoly}) by multiplying equations (\ref{pexr u}) and (\ref{bpexpr u}) together and rearranging. If $e^{i\theta}=\pm 1$, then (\ref{factor1}) becomes
    \[\omega(\omega v^2 - 1)(\omega v^2 + 1)(\overline{z}\omega v^2 - 1)p - (\overline{z}\omega^2 v^4 - z)(\omega qv^2 - 1)=\pm(\overline{z} - z)(q + 1) = 0.\]
    Thus $z$ must be real. If one plugs in $e^{i\theta}=\pm 1$ and real $z$ into (\ref{fundamentalpoly}), then the equation holds.
\end{proof}
\begin{proposition}\label{omegavprop}
The unimodular parameters $\omega$ and $v$ can be written in terms of $e^{i\theta},p$ and q only:
\begin{align}
&\omega =-\frac{|p|^2e^{2i\theta} - |p|^2 \pm |p|\sqrt{|p|^2e^{4i\theta}- 4qe^{3i\theta} + (4- 2|p|^2 + 4q^2)e^{2i\theta}   - 4qe^{i\theta} + |p|^2}}{2p(q - e^{i\theta})},\label{omegaexpr}\end{align}
\begin{align}
v^2=-\frac{e^{i\theta}2p(q - e^{i\theta})}{|p|^2e^{2i\theta} - |p|^2 \pm|p|\sqrt{p^2e^{4i\theta} - 4qe^{3i\theta}+ (4- 2p^2 + 4q^2)e^{2i\theta}    -4qe^{i\theta}+p^2}},\label{vexpr}
\end{align}
where the sign of the square root match and depend on the specific choice of $e^{i\theta}$.
\end{proposition}
\begin{proof}
Solving equation (\ref{systemcon4}) for $\overline{z}$
gives
\begin{align}\label{bzexpr}
  \overline{z} =\frac{ -(\overline{p}\omega v^4 - \overline{p}v^2z - \omega p + z)}{\omega^2v^2(-v^2 + p)}.
\end{align}
Substituting in this value for $\overline{z}$ into the expression $r^2-\overline{r}^2=0$ (where $r^2$ is expressed as in (\ref{systemcon5})), one arrives at the equation
\[-\frac{q(|p|^2 - 1)(\omega p - z)(\omega v^2 - z)(|p|^2\omega^2v^4 - \overline{p}\omega qv^4 - \omega^2pv^2 + \overline{p}v^2 + \omega pq - |p|^2)}{\omega (|p|^2\omega qv^2 - \overline{p}qv^2z - \omega qv^2 + \overline{p}v^2 - |p|^2 + qz)(p-v^2 )(\omega^2pv^2 - \omega qv^2 - \omega pq + qz)}=0.\]
Note that $|p|\neq 1$, $|z|\neq 1$ and $|p|\neq |z|$, so we must have that
\[|p|^2\omega^2v^4 - \overline{p}\omega qv^4 - \omega^2pv^2 + \overline{p}v^2 + \omega pq - |p|^2=0.\]
Next, set $v^2=e^{i\theta}/\omega$ and substitute into the equation above. Simplifying gives
\begin{align}\label{uomegaeqn} 
|p|^2\omega e^{2i\theta} - \overline{p}qe^{2i\theta} + \omega^2pq - \omega^2pe^{i\theta} - |p|^2\omega + \overline{p}e^{i\theta}=0.\end{align}
Solving the equation directly above for $\omega$ gives the expression stated in (\ref{omegaexpr}).
The $\pm$ in the $\omega$ and $v$ expression must match since we need $\omega v^2= e^{i\theta}$. The sign of the square root depends on the particular $e^{i\theta}$. For example, when $e^{i\theta}=i$, the positive square root is the desired root. When $e^{i\theta}=-i$, the negative square root is the desired root. See section \ref{z=pmi} for additional details.
\end{proof}
In Kosi\'{n}ski, Thomas, and Zwonek's formulation, they assumed that the Blaschke product in the first component of $\phi$ (equation (\ref{deg2disk})) has a unimodular constant of 1. One may wonder if adding an extra unimodular constant to this factor simplifies the computations. The answer is as follows. Equation (\ref{fundamentalpoly}) contains $|p|$, but no instances of $p$ otherwise. Also, recall that (\ref{pexpr}) gives
\begin{align*}
p =\frac{(e^{i\theta}q - 1)(\overline{z}e^{2i\theta} - z)}{\omega(e^{2i\theta} - 1)(\overline{z}e^{i\theta} - 1)}, \text{ in particular } |p| =\Big|\frac{(e^{i\theta}q - 1)(\overline{z}e^{2i\theta} - z)}{(e^{2i\theta} - 1)(\overline{z}e^{i\theta} - 1)}\Big|.
\end{align*}
\noindent Thus rotating $p$ does not affect $e^{i\theta}=\omega v^2$. By rotating $p$, we  change the value of $\omega$ and $v$, but the value for $e^{i\theta}$ remains the same. Using this reasoning, one may take $c$ to be real valued, and rotate $p$ to satisfy these equations. There is no benefit from doing this, so we match the convention set in the \cite{KTZ17}. For the remainder of this section, assume $p$ is real.
\begin{lemma} Let $p,q>0$.
The polynomial in (\ref{fundamentalpoly}) is a real polynomial and may be rewritten as
\begin{align}\label{realpoly1}
   p^2\sin^2(\theta)(a-1)^2-((q-\cos\theta)^2+(1-p^2)\sin^2(\theta))b^2=0
\end{align}
where $ze^{-i{\theta}}=a+ib$ and $e^{i\theta}=\cos(\theta)+i\sin(\theta)$. Alternatively, we may rewrite (\ref{fundamentalpoly}) as
\begin{equation}\label{realpoly2}
    \begin{split}
    &p^2\sin^2(\theta)(x^2+y^2+1-2x\cos(\theta)-2y\sin(\theta))\\&-(q^2-2q\cos(\theta)+1)(-x\sin(\theta)+y\cos(\theta))^2=0
    \end{split}
\end{equation}
where $e^{i\theta}=\cos(\theta)+i\sin(\theta)$ and $z=x+iy$.
Further, this polynomial will always have at least two zeros.
\end{lemma}
\begin{proof}
Note that
\begin{align*}
    &(e^{i\theta} - z)(\overline{z} e^{i\theta} - 1)(e^{2i\theta} - 1)^2p^2 - (e^{i\theta} - q)(\overline{z} e^{2i\theta} - z)^2(qe^{i\theta} - 1) =0
    \\&\Longleftrightarrow -(1-ze^{-i\theta})(1-\overline{z} e^{i\theta})(e^{i\theta}-e^{-i\theta})^2p^2+(e^{i\theta}-q)(\overline{z}e^{i\theta}-ze^{-i\theta})^2(e^{-i\theta}-q)=0
    \\&\Longleftrightarrow  -|1-ze^{i\theta}|^2(2i\text{Im}(e^{i\theta}))^2p^2+|e^{i\theta}-q|^2((-2i\text{Im}(ze^{-i\theta}))^2)=0.
\end{align*}
Continuing on in this manner gives (\ref{realpoly1}) and (\ref{realpoly2}). When $\theta=0$, (\ref{realpoly2}) reduces to $-(q-1)^2y^2<0$. When $\theta=\pi$, then (\ref{realpoly2}), reduces to $-(q^2+2q+1)y^2<0$. Suppose that $x>0$, and let $\theta=\pi+\arctan(y/x)$. Then $-x\sin(\theta)+y\cos(\theta)=0$ and $x\cos(\theta)+y\sin(\theta)=-(x/|x|)\sqrt{x^2+y^2}=-\sqrt{x^2+y^2}$. Then (\ref{realpoly2}) reduces to
    \begin{align*}
        p^2\frac{(y/x)^2}{1+(y/x)^2}(x^2+y^2+1+2\sqrt{x^2+y^2})>0.
    \end{align*}
    If $x<0$, then repeat the work above with $\theta=\arctan(y/x)$ to show that the polynomial will be positive.
    Thus there will be a least one $\theta\in[0,2\pi)$ such that the polynomial in (\ref{realpoly2}) is equal to zero. Given that only $\sin^2(\theta)$ and $\cos(\theta)$ appear in (\ref{realpoly1}), then it follows that if $\theta$ is a solution, then so is $-\theta$.
\end{proof}

\begin{proposition}\label{zrealprop}
    Let $p>0$, then $z$ is real if and only if $e^{i\theta}=\pm 1$.
\end{proposition}
\begin{proof}
First assume $z$ is real. If $-p<z<p/q$, then $\omega=-1$ and $v=i$ as is shown in section \ref{formulazreal}. If $-p/q<z<-p$, then $\omega=-1$ and $v=1$ as shown in section \ref{formulazreal}. 
Next, assume $e^{i\theta}=\pm 1$, then $\theta=0$ or $\theta=\pi$. The polynomial in (\ref{realpoly2}) then reduces to
\begin{align*}
    -(q^2\pm2q+1)(\pm y).
\end{align*}
Thus $y=0$.
\end{proof}

\begin{proposition}\label{ParametersProp} Let $e^{i\theta}=\omega v^2$. Assume (\ref{blaschkecon1})-(\ref{blaschkecon5}),  $z\in \D\setminus (-1,1)$, and $|z|<p/q$.
Then the parameters can be written as follows:
\begin{align}
    &r^2=\frac{(e^{i\theta} - q)(\overline{z} - z)(\overline{z} e^{2i\theta}  - z)(e^{i\theta} q - 1)}{(\overline{z} q^2 e^{2i\theta}  - |z|^2q e^{2i\theta}  + \overline{z} e^{2i\theta}  - qe^{2i\theta}  - 2\overline{z} qe^{i\theta}  + 2q e^{i\theta} z + q|z|^2 - q^2z + q - z)(e^{2i\theta}  - 1)}\label{rsexpr}\\
    &t=\frac{q(\overline{z} qe^{2i\theta} - e^{2i\theta} - \overline{z}e^{i\theta} + e^{i\theta}z - qz + 1)(\overline{z}e^{i\theta}- 1)(e^{i\theta} - z)}{Q}\label{texpr}\\
    &\alpha=\frac{v}{e^{i\theta}r}\frac{e^{3i\theta}\overline{z}q-e^{3i\theta}r^2+r^2ze^{2i\theta}-e^{2i\theta}\overline{z}-e^{i\theta}qz+e^{i\theta}r^2-r^2z+z}{\overline{z}qe^{2i\theta}-e^{2i\theta}-\overline{z}e^{i\theta}+ze^{i\theta}-qz+1}\label{alphaexpr}\\
    &\beta=\frac{v}{e^{i\theta} r}\frac{|z|^2e^{i\theta}-r^2e^{i\theta}+r^2z-z}{|z|^2-1}.\label{betaexpr}
\end{align}
where
\begin{align*}
    &Q=e^{4i\theta}\overline{z}^2q^2 - e^{4i\theta}\overline{z}|z|^2q - e^{3i\theta}\overline{z}^2q + e^{3i\theta}\overline{z}|z|^2 - 2e^{3i\theta}\overline{z}q^2 + 2e^{3i\theta}|z|^2q - \overline{z}e^{3i\theta}+e^{3i\theta}q\\& + 3e^{2i\theta}\overline{z} q - 3e^{2i\theta}qz - 2 |z|^2 e^{i\theta}q  - |z|^2e^{i\theta}z + 2e^{i\theta}q^2z + e^{i\theta}qz^2 + |z|^2qz - e^{i\theta}q + e^{i\theta}z - q^2z^2.
\end{align*}
In particular, note that $|c|,t,|\alpha|,|\beta|$ only depend on $z,p,q$ and $e^{i\theta}=\omega v^2$ and no other instances of $\omega$ or $v$.
\end{proposition}
\begin{proof}
By (\ref{pexpr}) of proposition 2, we have
\begin{align*}
p =\frac{(e^{i\theta}q - 1)(\overline{z}e^{2i\theta} - z)}{\omega(e^{2i\theta} - 1)(\overline{z}e^{i\theta} - 1)}.
\end{align*}
Taking the conjugate of the expression above and simplifying gives
\begin{align}
\overline{p}=\frac{\omega(e^{i\theta} - q)(\overline{z}e^{2i\theta} - z)}{(e^{2i\theta} - 1)(e^{i\theta} - z)}.\label{bpexpr}
\end{align}
Equation (\ref{systemcon5}) of proposition \ref{systemsprop} states
 \begin{align*}
        r^2=\frac{p(1-\overline{p}v^2)(|p|^2\omega v^2 - \overline{z}\omega qv^2 + \overline{z}\omega pq - |p|^2z - \omega p + z)}{(|p|^2\omega v^2 - \overline{z}\omega q v^2 + \overline{z}\omega pq - |p|^2q - \omega p + q)(p-v^2 )}.
    \end{align*}
By substituting $rv$ for $c$, (\ref{pexpr}) for $p$, and $(\ref{bpexpr})$ for $\overline{p}$ into (\ref{systemcon5}), and simplifying (\ref{systemcon5}) gives (\ref{rsexpr}).
Next, again by proposition 1, we have
\begin{align}\label{texpragain}
    t=\frac{(p\overline{p} - 1)(|c|^2\overline{c}\omega p - \overline{c}p\overline{p}z - c|c|^2\omega + p\overline{p}c\omega - \overline{c}\omega p + cz)(\overline{c}-\overline{z}c\omega  )}{P},\end{align}

\noindent By substituting $rv$ for $c$, (\ref{pexpr}) for $p$, and $(\ref{bpexpr})$ for $\overline{p}$ into (\ref{texpragain}), and simplifying gives $(\ref{texpr})$. Repeating these substitutions into
\begin{align*}
&\alpha=\frac{\overline{c} |c^2|p - |c|^2c + cp\overline{p}  - \overline{c}p}{|c|^2(p\overline{p} - 1)}\text{ and }\beta=\frac{\overline{c}|c|^2z - c|c|^2\omega + |z|^2c\omega- \overline{c}z}{\omega |c|^2(|z|^2 - 1)},
\end{align*}
gives (\ref{alphaexpr}) and (\ref{betaexpr}).
\end{proof}

\begin{proposition}\label{(0,z)formulaprop}
Assume that $(0,z)$ is in region 2 $($that is: $|z|<p/q)$, $z\in\D\setminus (-1,1)$, and $p$ and $q$ are real. Then
    \begin{equation}\begin{split}
    &g_{\D^2}((0,z),(p,0),(0,q))=\log|cm_\gamma(c)|\\&=\log\Big|\frac{(e^{i\theta}q - 1)(\overline{z}e^{2i\theta} - z)(q-z)}{\overline{z}q^2e^{2i\theta} - |z|^2qe^{2i\theta} + \overline{z}e^{2i\theta}- qe^{2i\theta} - 2\overline{z} qe^{i\theta} + 2 qze^{i\theta} + q|z|^2 - q^2z + q - z}\Big|,\label{cdexpr}
    \end{split}\end{equation}
    where $e^{i\theta}$ is the unique unimodular constant such that:
    \end{proposition}
    \begin{align}
         (e^{i\theta} - z)(\overline{z} e^{i\theta} - 1)(e^{2i\theta} - 1)^2p^2 - (e^{i\theta} - q)(\overline{z} e^{2i\theta} - z)^2(qe^{i\theta} - 1)=0.\label{zreq1}
        \end{align}
        \begin{align}
        0<t=\frac{q(\overline{z} qe^{2i\theta} - e^{2i\theta} - \overline{z}e^{i\theta} + e^{i\theta}z - qz + 1)(\overline{z}e^{i\theta}- 1)(e^{i\theta} - z)}{Q}<1\label{zreq2}
        \end{align}
        
        \begin{align}\max\{|z|,p\}<|c|<1
        \text{ where } \label{zreq3}\end{align}
        \begin{align}
        |c|^2=r^2=\Big|\frac{(e^{i\theta} - q)(\overline{z} - z)(\overline{z} e^{2i\theta}  - z)(e^{i\theta} q - 1)}{(\overline{z} q^2 e^{2i\theta}  - |z|^2q e^{2i\theta}  + \overline{z} e^{2i\theta}  - qe^{2i\theta}  - 2\overline{z} qe^{i\theta}  + 2q e^{i\theta} z + q|z|^2 - q^2z + q - z)(e^{2i\theta}  - 1)}\Big|,\label{zreq4}
    \end{align}
    and $Q$ is defined in proposition 5 ($Q$ depends only on $z,p,q$ and $e^{i\theta}$).
\begin{proof}
As was done previously, pick $(0,z)$ in region 2 and consider the triple of points
$(0,0),(p,z),$ and $(0,m_z(q))$ in region 2 with $z\in \D\setminus{(-1,1)}$. Lemma 3 of \cite{KTZ17} show that there exists $\alpha,\beta,c\in \D$, $t\in (0,1)$, and $\omega \in \mathbb{T}$ that satisfy conditions (\ref{blaschkecon1})-(\ref{blaschkecon5}). Propositions 1 and 2 then concludes that there must exist $e^{i\theta}=\omega 
v^2$ such that
 \[(e^{i\theta} - z)(\overline{z} e^{i\theta} - 1)(e^{2i\theta} - 1)^2|p|^2 - (e^{i\theta} - q)(\overline{z} e^{2i\theta} - z)^2(qe^{i\theta} - 1)=0.\]
 Now, there exists more then one unimodular constant that satisfies this equation for a given $z,p,$ and $q$. Finding a unimodular constant that satisfies (\ref{fundamentalpoly}) only guarantees that the expressions for $|c|$ and $t$ are real valued. However, we need $c,\alpha,\beta \in \D$ and $t\in(0,1)$. Kosi\'{n}ski, Thomas, and Zwonek \cite{KTZ17} show that the following map $\Phi:\D^3\times \mathbb T\times (0,1)\rightarrow \D^4$, 
\[\Phi(\alpha,\beta,c,\omega,t)=(cm_{\alpha}(c),\omega c m_{\beta}(c),m_\gamma(c)m_{\alpha}(m_\gamma(c)),\omega m_{\gamma}(c) m_{\beta}(m_\gamma(c))),\]
$\gamma=\alpha t+(1-t)\beta$, is a 2 to 1 map onto region 2 with $\Phi(-\alpha,-\beta,-c,\omega,t)=\Phi(\alpha,\beta,c,\omega,t)$. This proves there will be a unimodular constant $e^{i\theta}$ that satisfies (\ref{fundamentalpoly}) and so that $c,\alpha,\beta\in \D$ and $t\in (0,1)$. Next, note that $\max\{p,|z|\}<|c|<1$ implies that $\alpha,\beta\in\D$ because 
\[m_{\alpha}(c)=\frac{p}{c}, \hspace{3mm} m_{\beta}(c)=\frac{z}{c}.\]
So there will be a unimodular constant $e^{i\theta}$ such that the conditions stated in this proposition hold. Further, the paper \cite{KTZ17} shows that $c=rv$ (up to a sign) and $\omega$ are unique, so there will exactly one such $\omega v^2=e^{i\theta}$. Recall that $g_{\D^2}$ is equal to $\log|cm_{\gamma}(c)|=\log|c\alpha|$ because $(0,z)$ is in region 2 and $m_{\gamma}(c)=\alpha$. The expression for $g_{\D^2}$ stated in (\ref{cdexpr}) is found by substituting in the the expression for $r$ given in (\ref{zreq4})  into (\ref{alphaexpr}), and then multiplying this expression by (\ref{zreq4}) and simplifying.
\end{proof}
\noindent\textbf{Remark:}
The fact that the parameters $t$ and $r^2$ are real can be seen directly by substituting in $z=x+iy$ and $e^{i\theta}=\cos(\theta)+i\sin(\theta)$. Using Maple to simplify equations (\ref{zreq4}) and (\ref{zreq2}) with these substitutions give the expressions below. We have $|c|^2$ is
\begin{align*}
    &|c|^2=r^2=A_1/B_1,
    \end{align*}
    where
    \begin{align*}
        A_1=y(1-2\cos(\theta)q + q^2)(\cos(\theta)y - \sin(\theta)x)
    \end{align*}
    and 
    \begin{align*}
        &B_1=\sin(\theta)\cos(\theta)q^2y + \cos^2(\theta)q^2x - \cos^2(\theta)qx^2 - \cos^2(\theta)qy^2 + y\sin(\theta)\cos(\theta) \\&- 2\sin(\theta)qy - \cos^2(\theta)q + \cos^2(\theta)x - q^2x + qx^2 + qy^2 + q - x.
    \end{align*}
    Likewise a real expression for $t$ is given by
    \begin{align*}
    &t=A_2/B_2
\end{align*}
where
\begin{align*}
    &A_2=q(4\cos^2(\theta)qxy - 2\cos(\theta)\sin(\theta)qx^2 + 2\cos(\theta)\sin(\theta)qy^2 - \cos(\theta)qx^2y - \cos(\theta)qy^3\\& + \sin(\theta)qx^3 + \sin(\theta)qxy^2 - 2y\cos^2(\theta) + 2\cos(\theta)\sin(\theta)x - \cos(\theta)qy - 2\cos(\theta)xy+ \sin(\theta)qx\\&  - \sin(\theta)x^2 - 3\sin(\theta)y^2 - 2yqx + yx^2 + y^3 - \sin(\theta) + 3y)
\end{align*}
and
\begin{align*}
    &B_2=4yq^2\cos^2(\theta)x - 2yq\cos^2(\theta)x^2 - 2y^3q\cos^2(\theta) - 2\cos(\theta)\sin(\theta)q^2x^2 + 2\cos(\theta)\sin(\theta)q^2y^2 \\&+ 2\cos(\theta)\sin(\theta)qx^3 + 2\cos(\theta)\sin(\theta)qxy^2 - 2\cos(\theta)q^2y - 2\cos(\theta)qxy + \cos(\theta)x^2y + \cos(\theta)y^3 \\&+ 2\sin(\theta)q^2x - \sin(\theta)qx^2 - 3\sin(\theta)qy^2 - \sin(\theta)x^3 - \sin(\theta)xy^2 - 2yq^2x + yqx^2 + y^3q - \cos(\theta)y \\&- q\sin(\theta) + \sin(\theta)x + 3yq.
\end{align*}

\noindent For every $|z|<\min\{p/q,1\}$ there will be a unique $e^{i\theta}$ as in the proposition above.
\begin{definition} 
Define
\begin{align*}
    &A(1):=\{x\in (-1,1):\max\{-1,-p/q\}<x<-p\}\\&
    A(-1):=\{x\in(-1,1):-p<x<\min\{1,p/q\}\}\\&
    A(e^{i\theta}):=\{z\in\D\setminus (-1,1): |z|<p/q, (\ref{zreq1}), (\ref{zreq2}), (\ref{zreq3} ) \} \text{ for } e^{i\theta}\neq \pm1.
\end{align*}
\end{definition}
\noindent Note that all $(0,z)$ in region two will fall into a unique $A(e^{i\theta})$.

\begin{lemma}\label{Alem}
If $A(e^{i\theta})$ is nonempty, then it is a single line segment in the plane. 
\end{lemma}
\begin{proof}
Fix $\theta$. If $e^{i\theta}=\pm 1$, then the lemma follows immediately. Assume $e^{i\theta}\neq \pm 1$. If $A(e^{i\theta}$) is nonempty, it follows from equation (\ref{realpoly1}) that $A(e^{i\theta})$ is at most the union of two line segments passing through $(1,0)$. Set $\zeta:=ze^{-i\theta}$. Then equation  (\ref{realpoly1}) becomes
\begin{align*}
    \omega = \frac{(e^{i\theta} q  - 1)(\overline{z}e^{2i\theta} - z)}{ p(e^{2i\theta} - 1)(\overline{z}e^{i\theta} - 1)}=\frac{e^{i\theta}(e^{i\theta}q-1)(\overline\zeta-\zeta)}{p(e^{2i\theta}-1)(\overline\zeta-1)}.
\end{align*}
Now, the two lines provided by equation (\ref{realpoly1}) are symmetric about the $x$ axis. Thus if $\zeta$ is on one line segment, then $\overline\zeta$ is on the other line segment. However, we know that $\omega$ is constant when $e^{i\theta}$ is fixed. Thus
\begin{align*}
    \frac{e^{i\theta}(e^{i\theta}q-1)(\overline\zeta-\zeta)}{p(e^{2i\theta}-1)(\overline\zeta-1)}=\frac{e^{i\theta}(e^{i\theta}q-1)(\zeta-\overline\zeta)}{p(e^{2i\theta}-1)(\zeta-1)}
\end{align*}
However, one can compute that 
\begin{align*}
    \frac{e^{i\theta}(e^{i\theta}q-1)(\overline\zeta-\zeta)}{p(e^{2i\theta}-1)(\overline\zeta-1)}-\frac{e^{i\theta}(e^{i\theta}q-1)(\zeta-\overline\zeta)}{p(e^{2i\theta}-1)(\zeta-1)}=0 \Longleftrightarrow (\overline{\zeta}-\zeta)(\zeta-2+\overline{\zeta})(e^{i\theta}q-1)=0.
\end{align*}
Note that $(\overline{\zeta}-\zeta)(\zeta-2+\overline{\zeta})(e^{i\theta}q-1)\neq 0$ unless $\zeta$ is real, $|\zeta|\geq 1$, or $q$ is unimodular. But $\zeta$ is not real unless $z$ is real. Thus $A(e^{i\theta})$ cannot contain points on both lines.

\end{proof} 

\subsection{An Explicit formula when $z$ is real valued.}\label{formulazreal}

Let $z\in (-1,1)$ and $p,q>0$. Recall that $(0,z)$ is in region 2 if $|z|<p/q$. If $p\geq q$, then $p/q\geq 1$ and all $(0,z)$ are in region 2. In this case we have
\begin{align}
    g_{\D^2}((0,z),(p,0),(0,q))=
     \log\max\Big\{\Big| \frac{p(q - z)}{pq + qz + p + q}\Big|,\Big| \frac{p(q - z)}{-pq - qz + p + q}\Big|\Big\}.
\end{align}
If $p<q$, then we have some of region 1.
\begin{align}\label{zrealformula}
    g_{\D^2}((0,z),(p,0),(0,q))=\begin{cases} \log\Big|z\frac{q-z}{1-qz}\Big|& p/q\leq |z|<1\\
     \log\max\Big\{\Big| \frac{p(q - z)}{pq + qz + p + q}\Big|,\Big| \frac{p(q - z)}{-pq - qz + p + q}\Big|\Big\} & 0\leq |z|<p/q
    \end{cases}
\end{align}
These two formulas are derived in the next two sections. \\
\\
\subsubsection{\large The Case When $v=i$ and $\omega =-1$}
 Choosing $v=i$ and $\omega=-1$ and simplifying (\ref{systemcon5}) gives
\begin{align}\label{rsexprzreal1}
|c|^2=\frac{(-pz - qz + p + z)p}{(-pq - qz + p + q)},\hspace{3mm} \max\{p,|z|\}<|c|<1.
\end{align}
We need to check that $\max\{p^2,|z|^2\}<|c|^2<1$. It is clear that $(-pq - qz + p + q)>0$ for $-1<z<1$. Thus
\[|c|^2>p^2\Longleftrightarrow (-pz - qz + p + z)p-p^2(-pq - qz + p + q)>0\Longleftrightarrow p(p - 1)(p +z)(q-1)>0,\]
\[|c|^2>z^2\Longleftrightarrow (-pz - qz + p + z)p-z^2(-pq - qz + p + q)>0\Longleftrightarrow (1-z)(p+z)(p-zq)>0,\]
\[|c|^2<1\Longleftrightarrow (-pz - qz + p + z)p-(-pq - qz + p + q)<0\Longleftrightarrow (1-z )(p -1)(p+q)<0.\]
For the three inequalities above to be satisfied, we need $-p<z<p/q$.
Next, compute $t$ using (\ref{systemcon1}):
\begin{align}\label{texprzreal1}
t=\frac{q(z - 1)(p - 1)}{(-2pq +pz - qz + p + q)},\hspace{3mm}0<t<1.\end{align}
One may check that the denominator of $t$ will always be positive when $-p<z<\min\{1,p/q\}$.  Thus
\[t<1 \Longleftrightarrow q(z - 1)(p - 1)-(-2pq +pz - qz + p + q)<0\Longleftrightarrow p(z + 1)(q - 1)<0. \]
Thus $0<t<1$ when $-p<z<\min\{1,p/q\}$.
We see this case corresponds to $A(1)$. The value of the pluricomplex Green function is
\[g_{\D^2}((0,z),(p,0),(0,q))=\log|c\alpha|=\log\Big|\frac{p(q - z)}{-pq - qz + p + q}\Big|.\]

\subsubsection{\large The Case When $v=1$ and $\omega=-1$}
Choose $v=1$ and $\omega=-1$ and 
repeat as before. One finds that when $-p/q<z<-p$, then (\ref{zreq1}), (\ref{zreq2}), and (\ref{zreq3}) hold.
The value of the Green function is
\[g_{\D^2}((0,z),(p,0),(0,q))=\log|c\alpha|=\log\Big| \frac{p(q - z)}{pq + qz + p + q}\Big|.\]
These two family of disks corresponds to the case $z\in A(1)$ and $z\in A(-1)$. Lastly, note that $pq+qz+p+q\leq -pq-qz+p+q\Longleftrightarrow 2q(p+z)\leq 0\Longleftrightarrow z<-p$. Hence the Green function is the maximum of the two derived expressions as stated above.\\ \\
\textbf{Remark:} Formula (\ref{uomegaeqn}) implies that if $\omega=1$, then $e^{i\theta}$ is $\pm 1$. But then proposition (\ref{zrealprop}) states that $z$ must be real. The calculations in this section shows that $\omega=-1$ when $z$ is real. So this implies that $\omega$ is never equal to 1. If $p$ or $q$ were not assumed to be real, then $\omega$ may equal 1.

\subsection{An Explicit Formula When $e^{i\theta}=\pm i$}\label{z=pmi}
Recall that we are assuming $p,q>0$.
\begin{proposition}
Let the polynomial below be the polynomial from equation (\ref{fundamentalpoly}):
\[P(z,\overline{z},e^{i\theta}):= (e^{i\theta} - z)(\overline{z} e^{i\theta} - 1)(e^{2i\theta} - 1)^2p^2 - (e^{i\theta} - q)(\overline{z} e^{2i\theta} - z)^2(qe^{i\theta} - 1).\]
 If
\begin{equation}\label{u=izreq}
\begin{split}
&p< |z|<\min\{p/q,1\},  
 \text{ Im}(z)>0, \text{ and } P(z,\overline{z},i)=0 \\&\text{  or equivalently  } p^2=\frac{(q^2+1)(\text{Re}(z))^2}{|z|^2+1-2\text{Im}(z)},
\end{split}
\end{equation}
then
\[g_{\D^2}((0,z),(p,0),(0,q))=\log\Big|\frac{(\overline{z} + z)(q - z)(q + i)}{2i\overline{z}q - 2iqz + \overline{z}q^2 - 2|z|^2q + q^2z + \overline{z} - 2q + z}\Big|.\]
The $z\in A(i)$ are of the form
\[z=\frac{(y - 1)p}{\sqrt{-p^2 + q^2 + 1}} + iy, \text{ with } p< |z|<\min\{p/q,1\},\hspace{1mm}  y>0, \]
and the formula for pluricomplex Green function may be restated as follows:
\begin{align*}
&g_{\D^2}((0,z),(p,0),(0,q))=\log\Big|\frac{p((yi-q)\sqrt{-p^2 + q^2 + 1}  + py - p)}{(p\sqrt{-p^2 + q^2 + 1} - qy + q)(i-q)}\Big|.
\end{align*}
Similarly, if
\begin{align}\label{u=-izreq}
p< |z|<\min\{p/q,1\}, \text{Im}(z)<0, \text{ and } P(z,\overline{z},-i)=0 \text{  or equivalently  } p^2=\frac{(q^2+1)(\text{Re}(z))^2}{|z|^2+1+2\text{Im}(y)},\end{align}
then
\[g_{\D^2}((0,z),(p,0),(0,q))=\log\Big|\frac{(\overline{z} + z)(q - z)(q - i)}{-2i\overline{z}q + 2iqz + \overline{z}q^2 - 2|z|^2q + q^2z + \overline{z} - 2q +z}\Big|.\]
The $z\in A(-i)$ are of the form
\[z=-\frac{(y+1)p}{\sqrt{-p^2+q^2+1}}+iy \text{ with } p< |z|<\min\{p/q,1\},\hspace{1mm} y<0,\]
and the formula for $g_{\D^2}$ may be restated as follows
\[g_{\D^2}((0,0),(p,0),(0,q))=\Big|\frac{p((yi-q)\sqrt{-p^2 + q^2 + 1}  - py - p)}{(p\sqrt{-p^2 + q^2 + 1} + qy + q)(q + i)}\Big|.\]
% and either $P(z,\overline{z},i)=0$ or $P(z,\overline{z},-i)=0$, then
% \begin{align*}
% &g_{\D^2}((0,z),(p,0),(0,q))\\&=\log\max\Big\{\Big|\frac{p((yi-q)\sqrt{-p^2 + q^2 + 1}  + py - p)}{(p\sqrt{-p^2 + q^2 + 1} - qy + q)(-q + i)}\Big|,\Big|\frac{p((yi-q)\sqrt{-p^2 + q^2 + 1}  - py - p)}{(p\sqrt{-p^2 + q^2 + 1} + qy + q)(q + i)}\Big|\Big\}.
% \end{align*}
\end{proposition}
% When $e^{i\theta}=i$, the $z$ in the unit disk that satisfy  (\ref{fundamentalpoly}), (\ref{zreq2}), and (\ref{zreq3}) are of the form
% \[z=\frac{(y - 1)p}{\sqrt{-p^2 + q^2 + 1}} + iy\]
% with $y>0$ and
% \[p<|z|<\min\{p/q,1\}.\]
% Equation (\ref{cdexpr}) reduces to
% \begin{align*}
% &g_{\D^2}((0,z),(p,0),(0,q))=\log\Big|\frac{p((yi-q)\sqrt{-p^2 + q^2 + 1}  + py - p)}{(p\sqrt{-p^2 + q^2 + 1} - qy + q)(-q + i)}\Big|.
% \end{align*}
\noindent These formulas are derived in the next two sections.
\subsubsection{ \large The extremal disks for $e^{i\theta}=i$}
If $z\in A(i)$, then (\ref{fundamentalpoly}) must hold with $e^{i\theta}=i$. Solving for $p^2$ in (\ref{fundamentalpoly}) gives
the expression for $p^2$ stated in (\ref{u=izreq}).
 Setting $z=x+iy$ and solving for $x$ in the equation above gives
\[z=\frac{\pm(y - 1)p}{\sqrt{-p^2 + q^2 + 1}} + iy.\]
The case with the negative square root
does not correspond to any $z\in \D$. Indeed, some tedious computations show that $t\notin(0,1)$ when $|z|<1$ in this case, so choose the positive square root:
\begin{align}\label{zforthetai}
    z=\frac{(y - 1)p}{\sqrt{-p^2 + q^2 + 1}} + iy.
\end{align}
Note that $|z|<1$ when $(2p^2-q^2-1)/(q^2+1)<y<1$. Assume $0<y<1$. Next, compute $|c|^2$ and $t$
\begin{align*}
&|c|^2=\frac{yp\sqrt{-p^2 + q^2 + 1}}{p\sqrt{-p^2 + q^2 + 1} - qy + q},\hspace{2mm} \\&t=\frac{q(y - 1)(pq + \sqrt{-p^2 + q^2 + 1})}{qy\sqrt{-p^2 + q^2 + 1} + 2p^3 - 2q^2p - q\sqrt{-p^2 + q^2 + 1} - yp - p}.
\end{align*}
It is clear that the expression for $|c|^2$ will be positive when $y>0$ and negative when $y<0$. Next, we 
have
\[|c|^2<1 \Longleftrightarrow (y - 1)(p\sqrt{-p^2 + q^2 + 1} + q)<0.\]
This will always hold for $0<y<1$. Next, for $|c|^2>p^2$, we need
\[|c|^2>p^2 \Longleftrightarrow -p(p^2\sqrt{-p^2 + q^2 + 1} - pqy - y\sqrt{-p^2 + q^2 + 1} + pq)>0.\]
Thus we need
\begin{align*}
&p^2\sqrt{-p^2 + q^2 + 1} - pqy - y\sqrt{-p^2 + q^2 + 1} + pq<0 \\&\Rightarrow y > \frac{p(p\sqrt{-p^2 + q^2 + 1} + q)}{pq + \sqrt{-p^2 + q^2 + 1}}=\frac{p(q\sqrt{-p^2 + q^2 + 1} + p)}{q^2 + 1}.
\end{align*}
Given our parametrization for $z$, this is equivalent to $ |z|>p$.
% Next, note that
% \[ \frac{p(p\sqrt{-p^2 + q^2 + 1} + q)}{pq + \sqrt{-p^2 + q^2 + 1}}<1\Longleftrightarrow (p^2-1)(p\sqrt{-p^2+q^2+1}+q)<0,\]
% which is always satisfied when $p<1$.
Next, we have
\begin{align*}
&|c|^2>|z|^2\Longleftrightarrow(y - 1)\\& \cdot\frac{-pq^2y\sqrt{-p^2 + q^2 + 1} + q^3y^2 + p^3\sqrt{-p^2 + q^2 + 1} - 2p^2qy - yp\sqrt{-p^2 + q^2 + 1} + p^2q + qy^2}{(-p^2 + q^2 + 1)}>0.
\end{align*}
\noindent This implies that
\[\frac{p(q\sqrt{-p^2 + q^2 + 1} + p)}{(q^2 + 1)}<y<\frac{p(pq + \sqrt{-p^2 + q^2 + 1})}{q(q^2 + 1)}.\]
Given our parametrization for $z$, this is equivalent to $p<|z|<p/q$.
Next, note that the numerator of $t$ is clearly negative. One may verify that the denominator of $t$ is negative and $t<1$ when $y>0$ and $p<|z|<\min\{p/q,1\}$. So this disk is valid when
\[z=\frac{(y - 1)p}{\sqrt{-p^2 + q^2 + 1}} + iy,\hspace{1mm} y>0,\text{ and } p<|z|<\min\{p/q,1\}.\]
\noindent To compute $\omega$ and $v^2$, substitute in the expression for $z$ given in (\ref{zforthetai}) into equation (\ref{pexpr}) and solve for $\omega$. Then
\[\omega=\frac{-p - i\sqrt{-p^2 + q^2 + 1}}{i-q }\text{  and  } v^2=\frac{i(i-q )}{-p - i\sqrt{-p^2 + q^2 + 1}}.\]
Substituting in $e^{i\theta}=i$ into equation (\ref{omegaexpr}) and simplifying gives 
\[\omega=-\frac{p \pm i\sqrt{-p^2 + q^2 + 1}}{i-q }.\]
Thus choosing $\sqrt{-p^2+q^2+1}$ to be the positive root is the correct choice.
\subsubsection{\large The extremal disks for $e^{i\theta}=-i$}
Set $e^{i\theta}=-i$. 
Repeating as was done in the previous section, one computes
\begin{align}\label{zfortheta-i}
z=-\frac{(y+1)p}{\sqrt{-p^2+q^2+1}}+iy, \hspace{5mm} |c|^2= -\frac{yp\sqrt{-p^2+q^2+1}}{p\sqrt{-p^2+q^2+1}+qy+q}. 
\end{align}
One may check that $|z|<1$, $p<|z|<\min\{ p/q,1\}$, and $0<t<1$ when $\text{Im}(z)=y<0$ and $p<|z|<p/q$. Using the expression given in (\ref{cdexpr}) gives
\[|cm_{\gamma}(c)|=\Big|\frac{p((yi-q)\sqrt{-p^2 + q^2 + 1}  - py - p)}{(p\sqrt{-p^2 + q^2 + 1} + qy + q)(q + i)}\Big|.\]
\noindent To compute $\omega$ and $v^2$, substitute in the expression for $z$ given in (\ref{zfortheta-i}) into equation (\ref{pexpr}):
\[\omega=\frac{p - i\sqrt{-p^2 + q^2 + 1}}{i+q }\text{  and  } v^2=\frac{-i(i+q )}{p - i\sqrt{-p^2 + q^2 + 1}}.\]
Substituting in $e^{i\theta}=-i$ into equation (\ref{omegaexpr}) and simplifying gives 
\[\omega=\frac{p \pm i\sqrt{-p^2 + q^2 + 1}}{i+q }.\]
Thus choosing $\sqrt{-p^2+q^2+1}$ to be the negative root is the correct choice.

\section{A Family of Hypersurfaces Spanning Region 2}\label{region2section}
The following proposition provides a way to find the value of the pluricomplex Green function for arbitrary $(z_1,z_2)$ in region 2 along the disks constructed in section \ref{region2(0,z)section}.
\begin{proposition}\label{FormulaProp}
Let $\Omega$ be hyperconvex and fix two poles $p,q\in\Omega$. Suppose that $g_\Omega(z,p,q)=l_{\Omega}(z,p,q)$. Let $\phi:\D\rightarrow \Omega$ be holomorphic and $\phi(0)=z$, $\phi(c)=p$, $\phi(d)=q$, and $l_{\Omega}(z,p,q)=\log|cd|$. Then for any $\eta\in \text{Im}(\phi(\D))$ $($say $\phi(\zeta)=\eta)$, 
\[l_{\Omega}(\eta,p,q)=\log|m_c(\zeta)|+\log|m_d(\zeta)|.\]
\end{proposition}

\begin{proof}
Define $g(t)=\inf\{\log|cd|: f\in \text{Hol}(\D,\Omega), f(0)=t, f(c)=p, f(d)=q  \}$. 
Note that $g(\phi(t))$ is a subharmonic function  (because $l_\Omega$ is plurisubharmonic) from $\D$ to $[-\infty,0)$ and has logarithmic poles at $c$ and $d$. Thus we have that
\[g(\phi(\zeta))\leq \log|m_{c}(\zeta)|+\log|m_{d}(\zeta)|.\]
 Next, define
\[v(\zeta):=g(\phi(\zeta))-(\log|m_{c}(\zeta)|+\log|m_{d}(\zeta)|).\]
Note that $v$ is bounded above, and $g(\phi)-(\log|m_c|+\log|m_d|)$ is subharmonic on $\D\setminus\{c,d\}$. Thus $v$ extends to a subharmonic function on $\D$, see Theorem 2.7.1 \cite{MK91}. Next, note that $v(0)=0$. So by the maximum principle we have that $g(\phi(\zeta))=\log|m_{c}(\zeta)|+\log|m_{d}(\zeta)|$.
\end{proof}
\noindent For region 2 of the bidisk, the extremal analytic disks have 
degree two Blaschke products in both components. Thus these 
disks will intersect the $z_1$ and $z_2$ axis. Recall that $A(e^{i\theta})$ is the set of $|z|<\min\{p/q,1\}$ such that (\ref{fundamentalpoly}), $(\ref{zreq2})$, and $(\ref{zreq3}$) are satisfied for $e^{i\theta}$. By Lemma \ref{Alem}, each nonempty $A(e^{i\theta})$ will be a line segment in the $z$ plane. The extremal disk associated with each $z\in A(e^{i\theta})$ will be used to define a hypersurface, and these hypersurfaces span region 2 as shown in Proposition \ref{spanningProp}.

\subsection{A Few Preliminary Results.}
 First, note that the previous section constructs two disks that have identical image for each triple $(0,0)$, $(p,z)$, and $(0,m_{z}(q))$ in region 2. 
 These two disks are
 \[\phi_{z}^+(\lambda)=(\lambda m_\alpha(\lambda),\omega \lambda m_{\beta}(\lambda)) \hspace{5mm}\text{ and } \hspace{5mm}\phi_{z}^{-}(\lambda)=(\lambda m_{-\alpha}(\lambda),\omega \lambda m_{-\beta}(\lambda)).\] 
 Because these two disks have identical image, it does not matter which one is chosen. Next, we will prove the injectivity of the disks. Suppose that $(z_1,m_{z}(z_2))\in \phi_z(\D)$. Let $\zeta\in \D$ such that $\phi_z(\zeta)=(z_1,m_z(z_2))$. First, we need $\zeta m_{\alpha}(\zeta)=z_1$
or equivalently $\zeta\alpha+\overline{\alpha}\zeta z_1-z_1=\zeta^2$.
Second, we need $\omega\zeta m_{\beta}(\zeta)=m_{z}(z_2)$
or equivalently
$-\zeta\beta-\frac{m_{z}(z_2)}{\omega}\overline{\beta}\zeta+\frac{m_{z}(z_2)}{\omega}=-\zeta^2$.
Adding  these two equations together (and hence canceling the $\zeta^2$ on the right hand side) gives
\[\zeta\alpha+\overline{\alpha}\zeta z_1-z_1+-\zeta\beta-\frac{m_{z}(z_2)}{\omega}\overline{\beta}\zeta+\frac{m_{z}(z_2)}{\omega}=0.\]
Now solve for $\zeta$:
\begin{align}
    \zeta=\frac{z_1-m_{z}(z_2)/\omega}{\alpha-\beta+\overline{\alpha}z_1-\overline{\beta}m_{z}(z_2)/\omega}.\label{zetaexpr}\end{align}
   
\subsection{An Equation for $g_{\D^2}$ Along an Extremal Disk}
Recall that for $z\in \D$ such that $|z|<\min\{ p/q,1\}$, there is a corresponding extremal analytic disk $\phi_z:\D\rightarrow \D^2$
where $\phi_z(\lambda)=(\lambda m_{\alpha}(\lambda),\omega \lambda m_{\beta}(\lambda))$ with $\phi_z(c)=(p,z),\hspace{2mm} \phi_z(m_{\gamma}(c))=(0,m_{z}(q))$, and  
\[g_{\D^2}((0,z),(p,0),(0,q))=\log|cm_{\gamma}(c)|.\]
Given the data $p,q$ and $z$, there are parameters $\alpha,\beta,t,c=vr,$ and $\omega$ associated with the extremal disks. Proposition \ref{ParametersProp} expresses $|c|,t,|\alpha|,|\beta|$ in terms of the data and $e^{i\theta}=\omega v^2$. However, these expressions are long. For the purposes of this section, we will use shorter expressions for $c,\alpha$, and $\beta$ as explained below. Solving for $\overline{p}$ in equation (\ref{systemcon4}) gives 
\[\overline{p}=-\frac{(-\overline{z}\omega^2v^4 + \overline{z}\omega^2pv^2 - \omega p + z)}{v^2(\omega v^2 - z)}.\]
While we are assuming $p$ is real, using (\ref{systemcon4}) as is leads to simpler computations.
Plugging in this expression for $\overline{p}$ into (\ref{systemcon5}) and simplifying gives
\begin{align}\label{rexpromegav}
r=\sqrt{\frac{p\omega(\overline{z}\omega^2pv^2 - \overline{z}\omega qv^2 - \omega p + z)}{\omega^2pv^2 - \omega qv^2 - \omega pq + qz}}.\end{align}
Writing $c=rv$, and substituting the value for $r$ given in (\ref{rexpromegav}) into $\alpha$ and $\beta$ as defined in (\ref{systemcon2}) and (\ref{systemcon3}) gives
\begin{align}\label{alphaexprevomega}
    \alpha=\frac{(z-q)v}{(\overline{z}\omega^2pv^2 - \overline{z}\omega qv^2 - \omega p + z)}\sqrt{\frac{(\overline{z}\omega^2pv^2 - \overline{z}\omega qv^2 - \omega p + z)p\omega}{\omega^2pv^2 - \omega qv^2 - \omega pq + qz}}
\end{align}
and 
\begin{align}\label{betaexprevomega}
    \beta=\frac{-S}{(|z|^2 - 1)\omega^2p(\overline{z}\omega^2pv^2 - \overline{z}\omega qv^2 - \omega p + z)v}\sqrt{\frac{(\overline{z}\omega^2pv^2 - \overline{z}\omega qv^2 - \omega p + z)p\omega}{\omega^2pv^2 - \omega qv^2 - \omega pq + qz}},
\end{align}
where $S$ is the following polynomial
\begin{align*}
    &S=\overline{z}\omega^4p^2v^4 - \overline{z}\omega^3pqv^4 - |z|^2\omega^3pv^4 - |z|^2\omega^3p^2v^2 + |z|^2\omega^2qv^4 + 2|z|^2\omega^2pqv^2 - |z|^2\omega qv^2z\\& - \omega^3p^2v^2 + 2\omega^2pv^2z + \omega^2p^2z - \omega qv^2z - \omega pqz - \omega pz^2 + qz^2.
\end{align*}
\begin{proposition}
    Let $z\in \D$ and $|z|<p/q$. Let $\phi_z:\D\rightarrow \D^2$ be the associated extremal disk for the triple $(0,0)$, $(p,z)$, and $(0,m_z(q))$. Then $(z_1,m_z(z_2))\in  \phi_z(\D)$ if
    \begin{equation}\label{diskcondition}
\begin{split}
&
\Big(z_1-\frac{m_{z}(z_2)}{\omega}\Big)^2+\frac{m_{z}(z_2)}{\omega}(\alpha+z_1\overline{\alpha})^2+z_1\Big(\beta+\frac{m_{z}(z_2)}{\omega}\overline{\beta}\Big)^2\\&-(\alpha+z_1\overline{\alpha})\Big(\beta+\frac{m_{z}(z_2)}{\omega}\overline{\beta}\Big)\Big(z_1+\frac{m_{z}(z_2)}{\omega}\Big)=0.
\end{split}
\end{equation} 
where $\alpha$ and $\beta$ are parameters defined in terms of the data (given in (\ref{alphaexprevomega}) and (\ref{betaexprevomega})) and the two unimodular constants $\omega$ and $v$.
\end{proposition}
\begin{proof}
When $(z_1,m_z(z_2))\in\phi(\D)$, there exists a $\zeta\in \D$ such that $\phi(\zeta)=(z_1,m_{z}(z_2))$. In equation (\ref{zetaexpr}), we found an expression for $\zeta$. Plugging in this expression into $z_1=\zeta m_{\alpha}(\zeta)$ and simplifying gives (\ref{diskcondition}).
\end{proof}

 Next, we write (\ref{diskcondition}) in terms of the data, $\omega$ and $v$.  Take the expressions for $\alpha$ and $\beta$ given in (\ref{alphaexprevomega}) and (\ref{betaexprevomega}) and substitute into equation (\ref{diskcondition}); solve for $q$:
\begin{equation}\label{qdiskcondition}\begin{split}
       &q=A/B,\end{split}
    \end{equation}

    where
    \begin{align*}
        &A= p\omega (\overline{z}^2 \omega^4v^6z_1^2z_2 + \overline{z}\omega^4pv^6z_1 - 
        \overline{z}\omega^4v^6z_1^2 - 
        \overline{z}\omega^3v^6z_1z_2 - 2\overline{z}\omega^3pv^4z_1z_2 - 
        \overline{z}\omega^3v^4z_1^2z_2 \\&+ 2\omega^2v^4z_1z_2|z|^2 - \overline{z}\omega^2v^4z_1z_2^2 
        + \overline{z}\omega^2pv^2z_1z_2^2 - \omega^3pv^4z_1 + \omega^3v^4z_1^2 + 2\omega^2v^4z_1z_2 + 
        2\omega^2pv^2z_1z_2 \\&- \omega v^4zz_2  + \omega v^4z_2^2- 3\omega v^2zz_1z_2 - \omega pz_1z_2^2 + v^2z^2z_2  - v^2zz_2^2+ zz_1z_2^2),
        \\
        \\&B=\overline z^2\omega^4v^6z_1^2z_2  - \overline{z}^2\omega^3v^4z_1^2z_2^2 + \overline{z}\omega^4pv^6z_1 - \overline{z}\omega^4v^6z_1^2 + \overline{z}^2\omega^3pv^2z_1^2z_2^2 - \overline{z}\omega^3v^6z_1z_2\\& - 3\overline{z}\omega^3pv^4z_1z_2 + \overline{z}\omega^3v^4z_1^2z_2 - \overline{z}\omega^3pv^2z_1^2z_2  + |z|^2\omega^2v^4z_1z_2 + \overline{z}\omega^2v^4z_1z_2^2 + 2|z|^2\omega^2pv^2z_1z_2 \\&- |z|^2\omega^2v^2z_1^2z_2 - \overline{z}\omega^2pz_1^2z_2^2  - |z|^2\omega v^2z_1z_2^2 - \omega^2pv^4z + \omega^2pv^4z_2 + \omega^2v^4zz_1 + |z|^2\omega z_1^2z_2^2 \\&- \omega^2pv^2zz_1 + 2\omega^2pv^2z_1z_2 + \omega^2v^2zz_1^2 + \omega^2pz_1^2z_2 + \omega pv^2z^2 - \omega pv^2zz_2 - \omega v^2z^2z_1 - \omega v^2zz_1z_2\\& - \omega pzz_1z_2 - \omega zz_1^2z_2 + z^2z_1z_2).
    \end{align*}
Such a computation by hand is very strenuous. Rather using the algebraic manipulation capabilities of Maple allows for such computations to be done accurately. While expression (\ref{qdiskcondition}) is very long, it will be needed (in the next section) to find the formula for $g_{\D^2}$ along the image of $\phi_z$ when eliminating the parameter $z$. \par
\begin{proposition}
    Suppose that $(z_1,m_z(z_2)) \in \phi_z(\D)$ (satisfies (\ref{diskcondition})) where $\phi_z$ is the extremal disk for the triple $(0,0)$, $(p,z)$, and $(0,q)$ defined earlier. Then 
    \begin{align}\label{cdformulawithz}
    g_{\D^2}((z_1,z_2),(p,0),(0,q))=\log|P/Q|
    \end{align}
    where
    \begin{align*}
        &P=p(\overline{z}\omega v^2 - 1)(\overline{z}\omega^2v^2z_1z_2 + \omega^2pv^2 - \omega^2v^2z_1 - \omega v^2 z_2 - \omega pz_2 + zz_2)(\overline{z}\omega^2 v^2z_1z_2 - \overline{z} \omega qz_1z_2\\& - \omega^2v^2z_1 + \omega qv^2 - \omega v^2z_2 + \omega qz_1 - qz + zz_2)
    \end{align*}
    and
    \begin{align*}
        &Q=(\omega v^2 - z_2)(-\overline{z}^2\omega^3v^4z_1z_2 + \overline{z}^2\omega^3pv^2z_1z_2 - \overline{z}\omega^3pv^4 + \overline{z}\omega^3v^4z_1 - \overline{z}\omega^3pv^2z_1 + \overline{z}\omega^2v^4z_2\\& + \overline{z}\omega^2pv^2z - \overline{z}\omega^2pz_1z_2 - |z|^2\omega v^2z_2 + |z|^2\omega z_1z_2 + \omega^2pv^2 + \omega^2pz_1 - \omega v^2z - \omega pz - \omega zz_1 + z^2)\\&\cdot (\omega^2pv^2 - \omega qv^2 - \omega pq + qz).
    \end{align*}
\end{proposition}
\begin{proof}
 
We find the formula for $g_{\D^2}$ along $\phi_z(\D)$. Proposition \ref{FormulaProp} tells us that
\[g_{\D^2}(\phi(\zeta),(p,0),(q,0))=\log|m_c(\zeta)|+\log|m_d(\zeta)|,\]
where $c=rv$ is given by (\ref{rexpromegav}), $d=\alpha$ given by (\ref{alphaexprevomega}), and $\zeta$ is given by (\ref{zetaexpr}). Note that (\ref{zetaexpr}) gives $\zeta$ in terms of $\alpha$ and $\beta$, but $\alpha$ and $\beta$ are given in terms of the data, $\omega$, and $v$ in (\ref{alphaexprevomega}) and (\ref{betaexprevomega}). Thus we may write $\zeta$ in terms of the data, $\omega$, and $v$. Let $\phi(\zeta)=(z_1,m_{z}(z))$. Plugging in these expressions for $c,\alpha,$ and $\zeta$ into 
\[g_{\D^2}((z_1,z_2),(p,0),(0,q))=g_{\D^2}((z_1,m_z(z_2)),(p,z),(0,m_z(q)))=\log|m_c(\zeta)|+\log|m_d(\zeta)|,\]
and simplifying (using Maple's algebraic manipulation abilities) gives the expression stated in (\ref{cdformulawithz})
   \end{proof}
    \noindent This formula is not ideal as $z$ appears in the formula. The next section will eliminate the parameter $z$.
    \subsection{Eliminating the parameter $z$}\label{Elimz}
    The formula derived in the previous section along the disk $\phi_z$ includes the parameter $z$. We want to find a formula independent of this parameter. Note that equation (\ref{qdiskcondition}) is quadratic in $z$ and $\overline{z}$. Hence trying to solve for $z$ in this equation is not desirable. Rather, we will make a substitution so that the numerator and denominator of (\ref{qdiskcondition}) factors and one power of $z$ cancels. To do this, solve for $\overline{z}$ in equation (\ref{systemcon4}):
    \[\overline{z} =\frac{ -(p\omega v^4 - pv^2z - \omega p + z)}{\omega^2v^2(-v^2 + p)}.\]
        Substitute this expression for $\overline{z}$ into (\ref{qdiskcondition}) and factor. Both the numerator and the denominator will have a factor of $e^{i\theta}-z$. Hence, after cancellation, we get an expression that is linear in $z$ in both the numerator and denominator. Then solving for $z$ from this equation gives:
\begin{align}\label{zexpr}
    z = A/B
\end{align}
where
\begin{align*}
    &A=\omega p(-\omega^2p^2v^8z_1^2z_2 - \omega^3p^2v^8z_1 + \omega^3pv^8z_1^2 + \omega pqv^8z_1^2z_2 + \omega^3p^3v^6z_1 - \omega^3p^2v^6z_1^2 + \omega^2pqv^8z_1\\&+ \omega^2pv^8z_1z_2 - \omega^2qv^8z_1^2 - \omega^2p^2qv^6z_1  + \omega^2p^2v^6z_1z_2 + \omega^2pqv^6z_1^2+ \omega^2pv^6z_1^2z_2 - \omega qv^8z_1z_2 \\&- pqv^6z_1^2z_2^2 + \omega^3pv^6z_1 - \omega^3v^6z_1^2 - 2\omega^2p^3v^4z_1z_2 + \omega^2p^2v^4z_1^2z_2 - 2\omega pqv^6z_1z_2 + \omega pv^6z_1z_2^2  \\&+ \omega qv^6z_1^2z_2 + p^2qv^4z_1^2z_2^2 - \omega^3p^2v^4z_1 + \omega^3pv^4z_1^2 - 2\omega^2v^6z_1z_2+ 3\omega p^2qv^4z_1z_2 - 2\omega p^2v^4z_1z_2^2 \\&- 4\omega pqv^4z_1^2z_2 + qv^6z_1z_2^2 - \omega^2pqv^4z_1 + \omega^2pv^4z_1z_2 + \omega^2qv^4z_1^2 + \omega p^3v^2z_1z_2^2 + \omega p^2qv^2z_1^2z_2 \\&+ \omega qv^6z_2 - \omega v^6z_2^2  - pqv^4z_1z_2^2 + \omega^2p^2qv^2z_1 + \omega^2p^2v^2z_1z_2- \omega^2pqv^2z_1^2 - \omega^2pv^2z_1^2z_2 - 2\omega pqv^4z_2 \\&+ 2\omega pv^4z_2^2 + 3\omega qv^4z_1z_2 + pqv^2z_1^2z_2^2 + \omega p^2qv^2z_2 - \omega p^2v^2z_2^2 - 2\omega pqv^2z_1z_2 - p^2qz_1^2z_2^2 \\&- \omega p^2qz_1z_2 + \omega pqz_1^2z_2 - qv^2z_1z_2^2 + pqz_1z_2^2),\\ \\
&B=-\omega^2 p^3v^6z_1^2z_2  + \omega p^2qv^6z_1^2z_2 + 2\omega^2p^2v^6z_1z_2 - 2\omega^2p^3v^4z_1z_2  + 2\omega^2p^2v^4z_1^2z_2 - \omega pqv^6z_1z_2 \\&- p^2qv^4z_1^2z_2^2+ \omega^2pqv^6  - \omega^2pv^6z_2 - \omega^2qv^6z_1 - \omega p^2qv^4z_1z_2 - \omega pqv^4z_1^2z_2 + p^3qv^2z_1^2z_2^2  \\&- 2\omega^2p^2qv^4 + 2\omega^2p^2v^4z_2+ 2\omega^2pqv^4z_1 - 2\omega^2pv^4z_1z_2 + 2\omega p^3qv^2z_1z_2  - \omega p^2qv^2z_1^2z_2 + pqv^4z_1z_2^2\\& + \omega^2p^3qv^2 - \omega^2p^3v^2z_2 - \omega^2p^2qv^2z_1 + 2\omega^2p^2v^2z_1z_2 - \omega^2pv^2z_1^2z_2 + 2\omega qv^4z_1z_2 - p^2qv^2z_1z_2^2 \\&+ pqv^2z_1^2z_2^2 - \omega pqv^2z_1z_2 - p^2qz_1^2z_2^2 - \omega p^2qz_1z_2 +\omega pqz_1^2z_2 - qv^2z_1z_2^2 + pqz_1z_2^2.
\end{align*}
Now, we also have an expression for $\overline{z}$ by taking the conjugate of the expression (\ref{zexpr}). Substituting in these expressions for $z$ and $\overline{z}$ into (\ref{cdformulawithz}) and simplifying gives \footnotesize
    \begin{equation*}
\begin{split}
\\&g_{\D^2}((z_1,z_2),(p,0),(0,q))=\\&\log\Big|\frac{-\omega v^2 p^2z_1z_2 - \omega v^2pq + \omega v^2 pz_2 +\omega v^2qz_1 + p^2qz_1z_2 + \omega( p^2q -  p^2z_2 -  pqz_1 + p z_1 z_2) - qz_1z_2}{(\omega v^2)^2( p^2 z_1  -  p) + \omega( \omega v^2 p^2 - \omega v^2 p z_1- p^2q +  pqz_1) + v^2(\omega v^2 q  - \omega v^2 p q z_1+ p qz_1z_2  - qz_2)+pqz_2(1- pz_1)  }\Big|.
\end{split}
\end{equation*}
\normalsize
Note that the formula no longer depends on the parameter $z$. Rather, the formula depends on the two unimodular constants $\omega$ and $v$. Recall that $\omega v^2=e^{i\theta}$.  So after multiplying the denominator by $\omega$ (recall that $|\omega|=1$), the formula may be rewritten as follows:
\footnotesize
\begin{equation*}
\begin{split}
\\&g_{\D^2}((z_1,z_2),(p,0),(0,q))=\\&\log\Big|\frac{-e^{i\theta} p^2z_1z_2 -e^{i\theta}pq + e^{i\theta}pz_2 +e^{i\theta}qz_1 + p^2qz_1z_2 + \omega( p^2q -  p^2z_2 -  pqz_1 + p z_1 z_2) - qz_1z_2}{\omega e^{2i\theta}( p^2 z_1  -  p) + \omega^2( e^{i\theta} p^2 - e^{i\theta} p z_1- p^2q +  pqz_1) + e^{i\theta}(e^{i\theta} q  - e^{i\theta}p q z_1+ p qz_1z_2  - qz_2)+\omega pqz_2(1- pz_1)  }\Big|.
\end{split}
\end{equation*}
\normalsize
    \subsection{A Family of Hypersurfaces Spanning Region Two}\label{hypersurfacessection}
   
    As before, assume the poles are fixed at $(p,0)$, $(0,q)$ with $p,q>0$. Recall that for every $z\in \D$ in region two, that is $|z|<p/q$, there is a unique $e^{i\theta}$ (so that $z\in A(e^{i\theta}))$ and a corresponding analytic disk
\begin{align*}
    \phi_z(\lambda)=(\lambda m_{\alpha}(\lambda),\omega\lambda m_{\beta}(\lambda))
    \end{align*}
    with $\phi(0)=(0,0)$, $\phi_z(c)=(p,z)$, and $\phi_z(m_{\gamma }(c))=(0,m_{z}(q))$, where $\gamma=t\alpha+(1-t)\beta$, and the parameters, $\alpha=\alpha(z),\beta=\beta(z),c=c(z),\omega=\omega(z)$, and $t=t(z)$ are expressions in terms of the data and $e^{i\theta}$ as computed in Proposition \ref{ParametersProp}.
    For each $(0,z)$ in region 2, define the automorphism
    \begin{align*}
        G_z(\lambda_1,\lambda_2):=(\lambda_1,m_{z}(\lambda_2)).
    \end{align*}
 Next, for each nonempty $A(e^{i\theta}$), define \begin{align*}
S(e^{i\theta}) :=\bigcup_{z\in A(e^{i\theta})}G_z\circ \phi_z(\D).
\end{align*} Then each hypersurface $S(e^{i\theta})$ will contain the set $\{(0,z):z\in A(e^{i\theta})\}$ and the union of these hypersurfaces will span region 2 as shown in the proposition below.
    \begin{proposition}\label{spanningProp}
        For every $(z_1,z_2)$ in region 2, there is a $e^{i\theta}\in \mathbb T$ and a corresponding $z$ in region 2 that satisfies (\ref{zreq1}), (\ref{zreq2}), (\ref{zreq3}), and (\ref{qdiskcondition}), and $(z_1,z_2)\in S(e^{i\theta})$. Thus region two is the union of the hypersurfaces $S(e^{i\theta})$.
    \end{proposition}

    \begin{proof}
         Let $(z_1,z_2)$ be in region two, and let
         \[\phi(\lambda)=(B_1(\lambda),B_2(\lambda)):=(m_{z_1}(\lambda m_{\alpha}(\lambda)),m_{z_2}(\omega\lambda m_{\beta}(\lambda)))=( m_{z_1/d}(\lambda)m_{d}(\lambda),m_{z_2}(\omega\lambda m_{\beta}(\lambda)))\] where $B_1$ and $B_2$ are second degree Blaschke products and $\phi$ is an extremal disk of the form (\ref{deg2disk}) for $(z_1,z_2)$, i.e. $\phi(0)=(z_1,z_2)$, $\phi(c)=(p,0)$, $\phi(d)=(0,q)$, $d=m_{\gamma}(c)$, and $g_{\D^2}((z_1,z_2),(p,0),(0,q))=\log|cd|$. Then consider \[\tilde \phi(\lambda):=\phi( m_{z_1/d}(s\lambda))=(s\lambda m_d(m_{z_1/d}(s\lambda)),B_2(m_{z_1/d}(s\lambda)))\]
         where $s$ is a unimodular constant to be determined later. Then
         \[\tilde\phi(0)=(0,B_2(z_1/d)), \hspace{2mm}\tilde\phi(s^{-1}m_{z_1/d}(c))=(p,0), \hspace{2mm}\tilde\phi(s^{-1}m_{z_1/d}(d))=(0,q).\] 
         Proposition \ref{FormulaProp} tells us that
         \[g_{\D^2}((0,B_2(z_1/d)),(p,0),(0,q))=\log|m_c(z_1/d)|+\log|m_d(z_1/d)|.\]
         Thus $\tilde \phi$ is extremal for $(0,B_2(z_1/d))$, $(p,0)$ and $(0,q)$.
         Next, set $z:=B_2(z_1/d)$. Recall the automorphism $G_z(\lambda_1,\lambda_2)=(\lambda_1,m_{z}(\lambda_2))$. Then the disk $\Phi:=G_z\circ \tilde \phi$ is an extremal disk such that\[ \Phi(0)=(0,0), \hspace{2mm}\Phi(s^{-1}m_{z_1/d}(c))=(p,z),\hspace{2mm}\Phi(s^{-1}m_{z_1/d}(d))=(0,m_{z}(q)). \]
    Choose $s$ so that the first coordinate of $\Phi$ is monic. Then $\Phi$ is extremal (extremal disks are invariant under automorphisms), has second degree Blaschke products in both components, and is monic in the first component. The proof of Lemma 3 in 
    \cite{KTZ17} shows that such an extremal disk must be of the form (\ref{deg2disk}). Thus it must be one of the disks constructed in section \ref{generalComp}. Then note that $\Phi(s^{-1}z_1/d)=(z_1,m_z(z_2))$ and $G_z\circ \Phi(s^{-1}z_1/d)=\tilde\phi(s^{-1}z_1/d)=(z_1,z_2)$. Thus $(z_1,z_2)\in S(e^{i\theta})$.
    \end{proof}

    \subsection{A Formula for $g_{\D^2}((z_1,z_2),(p,0),(0,q))$ in Region 2.}
 
\begin{theorem}\label{completedFormulaThm}
Assume $p,q>0$. Region 2 is the union of the hypersurfaces $S(e^{i\theta})$, and if $(z_1,z_2)\in S(e^{i\theta})$, then 
\footnotesize
\begin{equation}\label{cdexprThm1}
\begin{split}
\\&g_{\D^2}((z_1,z_2),(p,0),(0,q))=\\&\log\Big|\frac{-e^{i\theta} p^2z_1z_2 -e^{i\theta}pq + e^{i\theta}pz_2 +e^{i\theta}qz_1 + p^2qz_1z_2 + \omega( p^2q -  p^2z_2 -  pqz_1 + p z_1 z_2) - qz_1z_2}{\omega e^{2i\theta}( p^2 z_1  -  p) + \omega^2( e^{i\theta} p^2 - e^{i\theta} p z_1- p^2q +  pqz_1) + e^{i\theta}(e^{i\theta} q  - e^{i\theta}p q z_1+ p qz_1z_2  - qz_2)+\omega pqz_2(1- pz_1)  }\Big|,
\end{split}
\end{equation}
\normalsize
where $\omega$ is a root of the polynomial
\[\omega pe^{2i\theta} + \omega^2q - \omega^2e^{i\theta} - qe^{2i\theta} - \omega p + e^{i\theta}=0.\]
In particular $\omega$ is:
\begin{align*}
&\omega =\frac{-pe^{2i\theta}+p\pm\sqrt{p^2e^{4i\theta} - 4qe^{3i\theta}+ (4- 2p^2 + 4q^2)e^{2i\theta}    -4qe^{i\theta}+p^2}}{2(q-e^{i\theta})}\end{align*}
where the sign of the square roots depends on the particular $S(e^{i\theta})$. Further $(z_1,z_2)\in S(e^{i\theta})$ if there is $|z|<p/q$ that satisfies (\ref{qdiskcondition}), and $z$ is related to $p,q,$ and $e^{i\theta}$ by (\ref{zreq1}), (\ref{zreq2}), and (\ref{zreq3}).
\end{theorem}
\begin{proof}
The expression for $g_{\D^2}((z_1,z_2),(p,0),(0,q))$ was found in section \ref{Elimz}. Recall that proposition \ref{omegavprop} states that when $z\in A(e^{i\theta})$, then $\omega$ is as stated. On the hypersurface $S(e^{i\theta})$, $e^{i\theta}$ is constant, so $\omega$ will be constant and will be as above. The conditions (\ref{zreq1}), (\ref{zreq2}), and (\ref{zreq3}) follow from Proposition \ref{(0,z)formulaprop}
\end{proof}
\noindent\textbf{Remark}: If one evaluates formula (\ref{cdexprThm1}) at $(0,z)$ and sets $p$ equal to the expression given in (\ref{pexpr}), then (\ref{cdexprThm1}) agrees with (\ref{cdexpr}).

\section{Applications of Theorem \ref{completedFormulaThm}}
In this section, we will use Theorem \ref{completedFormulaThm} to find an explicit formula for $g_{\D^2}$ for $S(\pm 1)$
 and $S(i)$. We will also explore the simplifications that occur in the special case $q=p$. As before assume $p,q>0$.
 
 \subsection{The formula along $S(\pm 1)$}\label{sectionhypersurface}
 \begin{corollary}\label{formulaS1}
Let $p,q>0$ and $(z_1,z_2)\in \D^2$. If there exists a real number $r$ such that $-p<r<\min\{1,p/q\}$
and 
\begin{align}\label{S1Req}
    r=\frac{p(pz_1z_2^2 + qz_1^2z_2 - 2pz_1z_2 - 2qz_1z_2 + pz_1 + qz_2 - z_1^2 + 2z_1z_2 - z_2^2)}{pqz_1^2z_2^2 - 2pqz_1z_2 - pz_1^2z_2 - qz_1z_2^2 + 2pz_1z_2 + 2qz_1z_2 + pq - pz_2 - qz_1},
\end{align}
then
\begin{align}\label{S1cdexpr}
    g_{\D^2}((z_1,z_2),(p,0),(0,q))=\log\Big| \frac{pqz_1z_2 - pz_1z_2 - qz_1z_2 - pq + pz_2 + qz_1}{pqz_1z_2 - pq - pz_1 - qz_2 + p + q}\Big|.
\end{align}
If there exists a real number $r$ such that $\max\{-1,-p/q\}<r<-p$ and \begin{align}\label{S-1Req}    
r=\frac{p(pz_1z_2^2 + qz_1^2z_2 + 2pz_1z_2 + 2qz_1z_2 + pz_1 + qz_2 - z_1^2 + 2z_1z_2 - z_2^2)}{pqz_1^2z_2^2 - 2pqz_1z_2 - pz_1^2z_2 - qz_1z_2^2 - 2pz_1z_2 - 2qz_1z_2 + pq - pz_2 - qz_1},
\end{align}
then
\begin{align} \label{S-1cdexpr}   
g_{\D^2}((z_1,z_2),(p,0),(0,q))=\log\Big|\frac{pqz_1z_2 + pz_1z_2 + qz_1z_2 - pq + pz_2 + qz_1}{pqz_1z_2 - pq - pz_1 - qz_2 - p - q}\Big|.\end{align}
 \end{corollary}
\begin{proof}
 Set $\omega=-1$ and $v=i$. Then $e^{i\theta}=1$. The expression for $z$ in (\ref{zexpr}) reduces to (\ref{S1Req}) when $\omega=-1$ and $v=i$. Section \ref{formulazreal} shows that we need $-p<z<\min\{1,p/q\}$.
Then for $(z_1,z_2)\in S(1)$, formula (\ref{cdexprThm1}) of Theorem \ref{completedFormulaThm} reduces to (\ref{S1cdexpr}) when $\omega=-1$ and $v=i$. Repeating these steps for $e^{i\theta}=-1$ gives us (\ref{S-1Req}) and (\ref{S-1cdexpr}).
\end{proof} 
 \noindent \textbf{Remark}: When $(z_1,z_2)=0$, (\ref{S1Req}) reduces to $z=0$ and (\ref{S1cdexpr}) reduces to $\log|pq/(p+q-pq)|$. This formula agrees with a partial formula found by Wikstr\"{o}m \cite{FW03}.
 \subsection{The formula along S(i)}
 For the hypersurface $S(i)$, we have that the $z\in A(i)$ are of the form
 \[p^2 = \frac{(q^2+1)\text{Re}(z)^2}{|z|^2+1-2\text{Im}(z)},\hspace{3mm}p<|z|<\min\{p/q,1\},\hspace{3mm} \text{Im}(z)>0.\]
 Equivalently, if one writes $z=x+iy$ and solves for $y$ above, then one gets the $z\in A(i)$ are of the form
 \[z=\frac{(y - 1)p}{\sqrt{-p^2 + q^2 + 1}} + iy,\]
with $y>0$ and $p<|z|<\min\{p/q,1\}$. Additionally, we have
\[\omega=\frac{-p - i\sqrt{-p^2 + q^2 + 1}}{i-q }\text{  and  } v^2=\frac{i(i-q )}{-p - i\sqrt{-p^2 + q^2 + 1}}.\]
See section \ref{z=pmi} for additional details. Substituting these values into (\ref{qdiskcondition}) gives a condition for $(z_1,z_2)\in S(i)$. Substituting in $\omega$ and $v$ as above into formula (\ref{cdexprThm1}) gives 
 \begin{equation*}
     \begin{split}
      &\log\Big|\Big((-p^2q+p^2z_2 +pqz_1-pz_1z_2) \sqrt{-p^2 + q^2 + 1}+2p^2qz_1z_2 +p^2q^2z_1z_2i - p^3z_2i - qz_1p^2i \\&- pqi + qz_1i + 2p^2qz_1z_2+pz_2i - q^2z_1z_2i + p^3qi + pq^2 - pqz_2 - q^2z_1 - z_1qz_2\Big)\Big/\Big((-p^2qz_1z_2\\&-2p^3+p^2z_1+pqz_2+p) \sqrt{-p^2 + q^2 + 1} +z_2qi+ - p^2q^2i - p^2z_2qi + q^2i + 2ip^4 - pqz_1z_2i\\& + pq^2z_1z_2 + pz_1i - p^3z_1i+ p^3qz_1z_2i - 2ip^2 - pqz_1 - q^2z_2 + q\Big)\Big|.
     \end{split}
 \end{equation*}
 \subsection{The $p=q$ case and the Carad\'eodory Metric on the Symmetrized Bidisk}\label{symmbidisksection}
 If $p=q$, equation (\ref{uomegaeqn})
 reduces to
 \[(\omega-1)(pe^{2i\theta}+p\omega-\omega e^{i\theta}-e^{i\theta})=0.\]
 This gives
 \begin{align}\label{omegavexprp=q}
 v^2 =\frac{ p-e^{i\theta}}{1-pe^{i\theta} } \hspace{5mm} \text{  and  }\hspace{5mm} \omega = \frac{e^{i\theta}(1-pe^{i\theta})}{p - e^{i\theta}}.\end{align}
 Additionally, region 1 is a subset of region 3.
 \begin{proposition}
    If $q=p>0$,  then region 1 is contained in region 3.
 \end{proposition}
\begin{proof} If $(z_1,z_2)$ lies in region one, then there exists an $f\in\text{Hol}(\D)$ such that
 \[f(z_1)=z_2 \text{ (or } f(z_2)=z_1) \text{ and } f(p)=0, f(0)=p.\]
 The invariant form of the Schwarz lemma states that for any $f\in \text{Hol}(\D)$ the following inequality
 \[\Big| \frac{f(\lambda_1)-f(\lambda_2)}{1-\overline{f(\lambda_2)}f(\lambda_1)}\Big|\leq \Big| \frac{\lambda_1-\lambda_2}{1-\overline{\lambda_2}\lambda_1}\Big|\]
 holds with equality if and only if $f$ is an automorphism of the unit disk.
When $\lambda_1=p$, $f(\lambda_1)=0$, $\lambda_2=0$, and $f(\lambda_2)=p$, this inequality implies that $f(\lambda)=m_p(\lambda)$. Thus $z_2=m_p(z_1)$. This implies that region 1 is contained in region 3.
\end{proof}
\noindent We can restate Theorem \ref{completedFormulaThm} for the $q=p$ case as follows. We group $z_1$ and $z_2$ together as $z_1+z_2$ and $z_1z_2$ in preparation for the following section on the symmetrized bidisk.
\begin{theorem}\label{p=qThm}
\text{ }
\\
\textbf{(a)} Suppose $q=p>0$. Let $(z_1,z_2)$ be in region 2. Then $(z_1,z_2)\in S(e^{i\theta})$ if there is a parameter $z\in \D$ such that
     \[\overline{z}pe^{4i\theta} - \overline{z}e^{3i\theta} - pe^{3i\theta} + pe^{i\theta} - pz + e^{i\theta}z=0,\]
    \[\max\{p^2,|z|^2\}<|c|^2=\frac{pe^{2i\theta}z - pe^{i\theta} + pz - e^{i\theta}z}{e^{i\theta}(pe^{2i\theta} - 2e^{i\theta} + z)}<1,\]
    \[0<t=\frac{(e^{i\theta} - z)(p - e^{i\theta})}{e^{i\theta}(pe^{2i\theta} + p - 2e^{i\theta})}<1,\]
   and
\begin{align}\label{zexprq=p}\begin{split}
    &z=A/B.
\end{split}\end{align}
where
\begin{align*}
        &A=-e^{i\theta}\big(p^2e^{4i\theta}z_1z_2 + e^{3i\theta}(-p^2z_1z_2(z_1+z_2) - p^2(z_1 +z_2) + p[(z_1+z_2)^2-4z_1z_2])\\&+ e^{2i\theta}(p^2(z_1z_2)^2 - [(z_1+z_2)^2-4z_1z_2] + 3p^2z_1z_2   + p(z_1 + z_2)) - 4pe^{i\theta}z_1z_2 + pz_1z_2(z_1 +z_2)  \\&- p^2(z_1z_2)^2\big),
    \\\\&B=e^{4i\theta}(p^2 - p(z_1+z_2))+ e^{3i\theta}(3pz_1z_2- p+ (z_1 + z_2))  -e^{2i\theta}( 4z_1z_2+ 2p^2z_1z_2)    + e^{i\theta}(3pz_1z_2 \\& - p(z_1z_2)^2    + z_1z_2(z_1 + z_2))+ p^2(z_1z_2)^2-pz_1z_2( z_1 +z_2).
\end{align*}
\textbf{(b)} If $(z_1,z_2)$ is in region 2 and $(z_1,z_2)\in S(e^{i\theta})$, then
\begin{align}\label{cdexprp=q}
    g_{\D^2}((z_1,z_2),(p,0),(0,p))=\log\Big|\frac{ pe^{2i\theta} - pz_1z_2 - e^{2i\theta}(z_1 + z_2) + 2e^{i\theta}z_1z_2}{pe^{2i\theta} - pz_1z_2 - 2e^{i\theta} + z_1 + z_2}\Big|.\end{align}
    \end{theorem}
\noindent\textbf{Remark}: If $(z_1,z_2)$ lies in region 1, then we have that $z_2=m_p(z_1)$. If one substitutes in $m_p(z_1)$ for $z_2$ into equation (\ref{cdexprp=q}), one gets that
    \[g_{\D^2}((z_1,z_2),(p,0),(0,p))=g_{\D^2}((z_1,m_p(z_1)),(p,0),(0,p))=\log\Big|z_1\frac{p-z_1}{1-pz_1}\Big|.\]
    This formula agrees with the formula found in region 1. Then by continuity, the formula also holds for region 3 and therefore the entire bidisk.
    \begin{proof}
         First, note that when $q=p$, (\ref{fundamentalpoly}) simplifies to
 \[(\overline{z}pe^{i\theta}- \overline{z}e^{2i\theta} + pe^{2i\theta} - pe^{i\theta}z - p + z)(\overline{z}pe^{4i\theta} - \overline{z}e^{3i\theta} - pe^{3i\theta} + pe^{i\theta} - pz + e^{i\theta}z) = 0.\]
 If the first factor is $0$, then $|c|=1$. This implies that (\ref{fundamentalpoly}) reduces to
 \begin{align}
     \overline{z}pe^{4i\theta} - \overline{z}e^{3i\theta} - pe^{3i\theta} + pe^{i\theta} - pz + e^{i\theta}z = 0.\label{p=qfundpoly}
 \end{align}
 Substituting in the expressions for $\omega$ and $v^2$ given in (\ref{omegavexprp=q}) into the expression for $z$ given by (\ref{zexpr}) and simplifying gives (\ref{zexprq=p}).
    Solving for $\overline{z}$ in equation (\ref{p=qfundpoly}), substituting into the expressions for $|c|^2$ and $t$, and simplifying gives the expressions for $|c|^2$ and $t$ stated in part (a). Formula (\ref{cdexprp=q}) is just formula (\ref{cdexprThm1}) with $\omega$ as in (\ref{omegavexprp=q}).
    \end{proof}
    
\subsubsection{The Carath\'eodory Metric for the Symmetrized Bidisk}
Let  
\begin{align*}
G:=\{ (z_1+z_2,z_1z_2): z_1,z_2\in\D\}
\end{align*}denote the symmetrized bidisk. The symmetrized bidisk is not convex, but it is polynomial convex, and has the  symmetrized torus as its
distinguished boundary, which is topologically a m\"obius band \cite{AY04}.  An important result by Lempert states the Kobayshi and 
Carath\'eodory distances are equal for any bounded convex domain in $\C^n$. Despite not being convex,  Agler and Young showed that Carath\'eodory and 
Kobayshi distances on $G$ are equal \cite{AY04}. Further, it was shown by C. Costrara that $G$ is not holomorphically equivalent to a convex domain \cite{CC04}. \par
A useful observation is that $G$ is the image of $\D^2$ under the ``symmetrization" map \begin{align*}
\pi(\lambda_1,\lambda_2)=(\lambda_1+\lambda_2,\lambda_1\lambda_2).
\end{align*}
This fact and the following proposition will allow us to find a formula for the Carath\'eodory metric using the formula for the pluricomplex Green function for the bidisk. A more general version of the following proposition was first proved by Edigarian and Zwonek  \cite{EZ98}.
\begin{proposition}[Edigarian and Zwonek \cite{EZ98}]\label{coveringProp}
    Let $\Omega_1\subset \C^n$ and $\Omega_2\subset \C^m$ be two domains, and let $F\in \text{Hol}(\Omega_1,\Omega_2)$ be a proper 2-1 map. Suppose that $z\in \Omega_1$, and $F(w_1)=F(w_2)=w$. Then
    \begin{align*}
        g_{\Omega_2}(F(z),w)=g_{\Omega_1}(z,w_1,w_2)
    \end{align*}
\end{proposition}
\begin{proof}
Define  $F^*g_{\Omega_2}(z,w):=g_{\Omega_2}(F(z),w)$. Then $F^*g_{\Omega_2}(z,w)\leq g_{\Omega_1}(z,w_1,w_2)$ because $F^*g_{\Omega_2}$ is negative, plurisubharmonic ($F$ is holomorphic), and has logarithmic poles at $w_1$ and $w_2$. Next, define the push forward $F_*u$ of $u\in PSH(\Omega_1)$ by
\[F_*u(w):=\max\{ u(z):z\in F^{-1}(w)\}.\]
$F_*$ will be PSH when $F$ is a proper map (see Prop 2.29.26 of \cite{MK91}). Let $u$ be in the defining family for $g_{\Omega_1}(\cdot, w_1,w_2)$. Then $F_*u$ is negative and has logarithmic poles at $w_1$ and $w_2$. Next, note that
\[F^*F_*(u)(z)=(F_*u)(F(z))=\max\{u(s):s\in F^{-1}(F(z))\}\geq u(z).\]
Thus
\[u(z)\leq (F^*F_*u)(z)\leq F^*g_{\Omega_2}(z,w)=g_{\Omega_2}(F(z),w).\]
Since $u$ was an arbitrary function in the defining function for $g_{\Omega_1}(\cdot,w_1,w_2)$, then it follows that $g_{\Omega_2}(F(z),w))\geq g_{\Omega_1}(z,w_1,w_2)$. Thus the equality follows.
\end{proof}
\noindent Next, note that $\pi(p,0)=(p,0)$ and $\pi(0,p)=(p,0)$. So, applying Proposition \ref{coveringProp} gives
\[g_{G}((z_1+z_2,z_1z_2),(p,0))=g_{\D^2
}((z_1,z_2),(p,0),(0,p)).\]
Now, one may use an  automorphism of $G$ to map an arbitrary point $(p_1+p_2,p_1p_2)\in G$ to $(p,0)$ with $p>0$. Indeed, let $h$ be an automorphism of $\D$, and define $A_{h}(\pi(\lambda_1,\lambda_2)):=\pi(h(\lambda_1),h(\lambda_2))$. The automorphisms of the symmetrized bidisk are of the form
\begin{align*}
\text{Aut}(G)=\{A_{h}:h\in \text{Aut}(\D)\}.\end{align*}
Let $p=m_{p_2}(p_1)$ and $h=(|p|/p)m_{p_2}(\lambda)$. Then
\begin{align*}
A_{h}(\pi(p_1,p_2))=\pi\big((|p|/p)m_{p_2}(p_1),0\big)=(|p|,0).
\end{align*}
Thus WLOG we may assume the pole is at $(p,0)$ with $p>0$. Agler and Young found an explicit formula for $C_{G}$ when one of the two points is zero \cite{AY01}:
\begin{align*}
    C_{G}((0,0),(s,p))=\tanh^{-1}\frac{2|s-p\overline{s}|+|s^2-4p|}{4-|s|^2}.
\end{align*} 
Agler and Young \cite{AY04} later found a complete formula:
\begin{align}\label{caraMetric}
C_{G}((s_1,p_1),(s_2,p_2))=\sup_{\tau \in \mathbb T}\Big|\frac{(s_2p_1-s_1p_2)\tau^2+2(p_2-p_1)\tau+s_1-s_2}{(s_1-\overline{s_2}p_1)\tau^2-2(1-p_1\overline{p_2})\tau+\overline{s_2}-s_1\overline{p_2}} \Big|,
\end{align}
where $s_1$ and $p_1$ are the sum and product of two complex numbers in $\D$ and $s_2$ and $p_2$ are the sum and product of two potentially different complex numbers in $\D$.  Recall formula (\ref{cdexprp=q}) found in the previous section:
\begin{align*}    
&g_{\D^2}((z_1,z_2),(p,0),(0,p))=\log\Big|\frac{ pe^{2i\theta} - pz_1z_2 - e^{2i\theta}(z_1 + z_2) + 2e^{i\theta}z_1z_2}{pe^{2i\theta} - pz_1z_2 - 2e^{i\theta} + (z_1 + z_2)}\Big|\\&=\log\Big|\frac{  pz_1z_2e^{-2i\theta}  - 2e^{-i\theta}z_1z_2+(z_1 + z_2)-p}{(z_1 + z_2)e^{-2i\theta}- pz_1z_2e^{-2i\theta}  - 2e^{-i\theta} +p}\Big|
\end{align*}
Now let $(s_1,p_1)=(z_1+z_2,z_1z_2)$ and $(s_2,p_2)=(p,0)$ with $p>0$, then (\ref{caraMetric}) becomes
\begin{align*}
    &\sup_{\tau \in \mathbb T}\Big|\frac{(s_2p_1-s_1p_2)\tau^2+2(p_2-p_1)\tau+s_1-s_2}{(s_1-\overline{s_2}p_1)\tau^2-2(1-p_1\overline{p_2})\tau+\overline{s_2}-s_1\overline{p_2}} \Big|=\sup_{\tau\in\mathbb{T}}\Big|\frac{pz_1z_2\tau^2-2z_1z_2\tau+(z_1+z_2)-p}{(z_1+z_2)\tau^2-pz_1z_2\tau^2-2\tau+p}\Big|.
\end{align*}
Thus, as expected, the formulas match: \begin{align*}
    \text{exp}(g_{\D^2}((z_1,z_2),(p,0),(0,p))=\sup_{\tau\in\mathbb{T}}\Big|\frac{pz_1z_2\tau^2-2z_1z_2\tau+z_1+z_2-p}{(z_1+z_2)\tau^2-pz_1z_2\tau^2-2\tau+p}\Big|.
\end{align*} Thus the unimodular constant that achieves the supremum in (\ref{caraMetric}) is a root of the polynomial
\begin{align*}
    \overline{z}pe^{4i\theta} - \overline{z}e^{3i\theta} - pe^{3i\theta} + pe^{i\theta} - pz + e^{i\theta}z = 0,
\end{align*}
where $z\in\D$ is the parameter described in Theorem \ref{p=qThm}.\par
Let $\Omega\subset \C^n$ be a domain in $\C^n$. The \textit{infinitesimal Carath\'eodory metric} (or sometimes called the \textit{Carathéodory–Reiffen psuedometric})
is the nonnegative function $\gamma$ defined on the tangent bundle of $\Omega$ by
\[\gamma_{\Omega}(z;v)=\sup\frac{|f'(z)v|}{1-|f(z)|^2}\]
where the supremum is taking over all holomorphic functions $f:\Omega\rightarrow \D$ and $f'(z)v=\sum_{k=1}^nv_k\frac{\partial f}{\partial z_k}(z)$.
For $z_1\neq z_2$, we have
\begin{align*}
    \lim_{\underset{\frac{z_1-z_2}{||z_1-z_2||}\rightarrow v}{z_1,z_2\rightarrow z}}\frac{C_{\Omega}(z_1,z_2)}{||z_1-z_2||}=\gamma_{G}(z;v),\hspace{1cm} z\in \Omega, v\in \mathbb{C}^n, ||v||=1.
\end{align*}
See \cite{JP93} page 59 for a proof of the equality above and chapter 2 for a detailed discussion of the infinitesimal Carath\'eodory metric. Agler and Young found the following formula for the infinitesimal Carathe\'eodory metric on the symmetrized bidisk \cite{AY04}:
\[\gamma_G(z,v)=\sup_{\tau \in \mathbb T}\Big| \frac{v_1(1-\tau ^2p)-v_2(2-\tau s)}{(s-\overline{s}p)\tau^2-2(1-|p|)^2\tau +\overline{s}-\overline{p}s}\Big|,\]
where $z=(s,p)\in G$ and $v=v_1 \frac{\partial}{\partial s}+v_2\frac{\partial}{\partial p}$.
In 2014, Trybula found a formula for $\gamma_{G}$ via Lagrange multipliers for the special case that the two points are $(0,p)$ and $z=(r_1,r_2e^{i\theta})$ where $p\in (0,1)$ and $(r_1,r_2e^{i\theta})\in \R_{\geq 0}\times \C$ \cite{MT14}. In particular, Trybula's formula is

\begin{align*}
\gamma_{G}((0,p), (r_1,r_2e^{i\phi})=
    \begin{cases}
    \frac{\sqrt{(p+1)^2|r_1|^2+(4+(1-p)^2/p)||r_2|^2}}{2(1-p^2)} & \text{if } pr_1r_2\neq 0, \sin \phi=0, b\leq 2
    \\ 
    \\\frac{\sqrt{\big[ 1+p^2-2p(2\lambda+1)+\frac{2pb^2}{(1-\lambda)^2}\big]|r_1|^2+4|r_2|^2}}{2(1-p^2)} & \text{if } pr_1p_2\neq 0, \sin \phi\neq 0, \\&\text{or if }pr_1p_2\neq 0, \sin\phi=0,b\geq 2 
    \end{cases}
\end{align*}
where $p\in(0,1)$, $r_1,r_2\geq 0$, $r_1r_2\neq 0$, and $\lambda$ is the only root of the 4th degree polynomial
\[x^4-x^2(2+a^2+b^2)+2x(a^2-b^2)+(1-a^2-b^2)=0\]
in $(-\infty,-1)$ and 
\begin{align*}
    a=\frac{r_2\sin\phi(p+1)}{pr_1},\hspace{3mm} b=\frac{r_2\cos\phi(1-p)}{pr_1}.
\end{align*}
We see here that a root of a fourth degree polynomial was required for an explicit formula in terms of the data mirroring the need for a root of a sixth degree polynomial in the formula for the pluricomplex Green function of the bidisk.
\section{Region 3}
Recall that region 3 is when $|m_{z_1}(p)|=|z_2|$, $|z_1|= |m_{z_2}(q)|$, $p=0$, or $q=0$. This is a closed set with empty interior. 
% For simplicity, assume that $(z_1,z_2)=(0,0)$, and the poles are $(p_1,p_2)$ and $(q_1,q_2)$. Then region 3 is $|p_1|=|p_2|$, $|q_1|=|q_2|$, $p_1=q_1$, or $q_2=p_2$.
\subsection{$p_1=q_1$ or $p_2=q_2$.}
It was previously known that 
\[g_{\D^2}((0,0),(p_1,p_2),(p_1,q_2))=\log\max\{|p_1|,|p_2q_2|\}.\]
Indeed, consider the holomorphic map $f:\D^2\rightarrow \D^2$ defined by
\[f(z_1,z_2)=\Big(\frac{z_1-p}{1-\overline{p}z_1},\frac{z_2-p_2}{1-\overline{p_2}z_2}\frac{z_1-q_2}{1-\overline{p_2}z_2}\Big).\]
Then $f$ is holomorphic and a proper 2-1 map. Then $f(p,p_2)=f(p,q_2)=0$, so by proposition \ref{coveringProp}, we have
\[g_{\D^2}((0,0),(p,p_2),(p,q_2))=g_{\D^2}((p,p_2q_2),(0,0))=\log\max\{ |p|,|p_1q_2|\}.\]
Thus
\[g_{\D^2}((z_1,z_2),(0,0),(0,q))=g_{\D^2}((0,0),(z_1,z_2),(z_1,m_{z_2}(q))=\log\max\{|z_1|,|z_2m_{q}(z_2))| \}.\]
\noindent One noteworthy property of this special case is that the extremal disk for the Lempert function may only pass through one pole. 
For example, suppose that $|z_2|<|z_1|$ or $|m_{z_2}(q)|<|z_2|$. Then there exists a function $f\in \text{Hol}(\D)$ that fixes the origin and $f(z_1)=z_2$ or $f(z_2)=|m_{z_2}(q)|$. For simplicity assume that $f(z_1)=z_2$. Next, define $\phi(\lambda)=(\lambda,f(\lambda))$. Then we have $\phi(z_1)=(z_1,z_2)$ and $\phi(0)=(0,0)$. Next, define $F(\zeta_1,\zeta_2)=m_{z_1}(\zeta_1)$. Then $F(z_1,z_2)=F(z_1,m_{q}(z_2))=0$ and $F(0,0)=z_1$. 
% Suppose that either $|p_2|<|p_1|$ or $|q_2|<|q_1|$. Then there exists a function $f\in H(\D)$ that fixes the origin and $f(p_1)=p_2$ or $f(p_2)=q_2$. For simplicity assume that $f(p_1)=p_2$. Next, define $\phi(\lambda)=(\lambda,f(\lambda))$. Then we have $\phi(p_1)=(p_1,p_2)$ and $\phi(0)=(0,0)$. Next, define $F(z_1,z_2)=m_{p_1}(z_1)$. Then $F(p_1,p_2)=F(p_1,q_2)=0$ and $F(0,0)=p_1$. 
\par
However, there are cases when the poles will fall into region 1 in the sense that there will exist an interpolating holomorphic function between the coordinates. If $|z_2|>|z_1|$, $|m_{z_q}(q)|>|z_1|$, and $|z_2m_{z_2}(q)|>|z_1|$, then there will be a holomorphic function $f:\D\rightarrow \D$ such that $f(0)=0$, $f(z_2)=z_1$, and $f(m_{q}(z_2))=z_1$. Then the extremal disk is $\phi(\lambda)=(f(\lambda),\lambda)$.

\subsection{$|p_1|=|p_2|$}
This case is more complicated. Indeed, consider the case when $z$ is real. By the continuity of the pluricomplex Green function and (\ref{zrealformula}), we have that
\[g_{\D^2}((0,p),(p,0),(0,q))=\log\Big|\frac{p(p-q)}{2pq - p - q}\Big|.\]
While on the other hand
\[g_{\D^2}((0,-p),(p,0),(0,q))=\log|p|.\]
The form of the extremal disks differ in these two cases. At $(0,p)$, one may compute (using the expressions found in section \ref{generalComp}) that
\[|c|=\sqrt{\frac{(p + q - 2)p^2}{(2pq - p - q)}} \hspace{2mm} \text{ and }\hspace{2mm}t=\frac{q(p - 1)^2}{(p^2 - 3pq + p + q)}.\]
However, when we consider $(0,-p)$, we get that 
\[|c|=p,\hspace{2mm} t=1=\beta, \hspace{2mm}\alpha=-1.\]
In this case the disks are of the form
\[\phi(\lambda)=(-\lambda,\lambda).\]

\section{Acknowledgments} The author would like to acknowledge all the guidance and fruitful discussion he had with Professor Dan Coman throughout the duration of this project. This work was partly supported by the NSF grant DMS-2154273 (Coman).

%%%%%%%%%% Insert bibliography here %%%%%%%%%%%%%%

\vskip2pc
%\noindent If maintaining .bib file for references, then please use "RS.bst" to generate the references.

%%%% .BST file

\bibliography{bib}
\bibliographystyle{plain}

\end{document}